\documentclass{article}
\usepackage{authblk}
\usepackage{amsthm}
\usepackage{mathtools}
\usepackage{thmtools}
\declaretheoremstyle[headfont=\normalfont]{normalhead}
\usepackage{hyperref}
\usepackage{amssymb}
\usepackage{enumerate}
\usepackage{todonotes}
\usepackage[round,square,sort,comma,numbers]{natbib}

\newcommand{\dist}{\mathrm{dist}_Z}

\newcommand\extrafootertext[1]{%
	\bgroup
	\renewcommand\thefootnote{\fnsymbol{footnote}}%
	\renewcommand\thempfootnote{\fnsymbol{mpfootnote}}%
	\footnotetext[0]{#1}%
	\egroup
}

\usepackage{float}

\usepackage[text={6.25in,10in}]{geometry}

\newcommand{\DeltaVAE}{$\Delta\!$VAE}

\title{Small time asymptotics of the entropy of the heat kernel on a Riemannian manifold}
\author[$\dagger$,$\star$]{Vlado Menkovski}
\author[$\dagger$,$\star$]{Jacobus W. Portegies}
\author[$\dagger$]{Mahefa Ratsisetraina Ravelonanosy}

\affil[$\dagger$]{\small Department of Mathematics and Computer Science\\ Eindhoven University of Technology}
\affil[$\star$]{\small EAISI, Eindhoven University of Technology}

	\theoremstyle{plain}
	\newtheorem{defn}[subsection]{Definition}
	\newtheorem{theorem}[subsection]{Theorem}
	\newtheorem{prop}[subsection]{Proposition}
	\newtheorem{lem}[subsection]{Lemma}
	
	\newtheorem{remark}[subsubsection]{Important remark}

\allowdisplaybreaks
\begin{document}
	\maketitle

	\begin{abstract}{}
\noindent
\hskip .1in
\textbf{Abstract.}
 We give an asymptotic expansion of the relative entropy between the heat kernel $q_Z(t,z,w)$ of a compact Riemannian manifold $Z$ and the normalized Riemannian volume for small values of $t$ and for a fixed element $z\in Z$. We prove that coefficients in the expansion can be expressed as universal polynomials in the components of the curvature tensor and its covariant derivatives at $z$, when they are expressed in terms of normal coordinates. We describe a method to compute the coefficients, and we use the method to compute the first three coefficients. The asymptotic expansion is necessary for an unsupervised machine-learning algorithm called the Diffusion Variational Autoencoder.
\end{abstract}

	\section{Introduction}
\extrafootertext{This work was supported by NWO GROOT project UNRAVEL, OCENW.GROOT.2019.044.}
In this article, we give a small-time asymptotic expansion of the relative entropy (also called the Kullback-Leibler divergence)
\begin{equation}\label{first_foremost_equation}
D_{KL}\left( q_Z(t,z,\cdot)\vert \vert \vartheta\right):=\int_Z q_Z(t,z,w)\log\left(q_Z(t,z,w)\right)d\mathrm{Vol}(w) + \log[\mathrm{Vol}(Z)]
\end{equation}
of the heat kernel $q_Z(t,z,\cdot)$ of a $d$-dimensional closed Riemannian manifold $Z$ with respect to the normalized Riemannian volume $\vartheta$. 

More precisely, our main result (Theorem \ref{Higher-expansion}) is that for every $n \in \mathbb{N}$ and $z \in Z$ we have the following $n$th order asymptotic expansion for small $t$
$$D_{KL}\left( q_Z(t,z,.)\vert \vert \vartheta\right) = -\frac{d}{2}\log(2\pi t) + \log[\mathrm{Vol}(Z)] +\sum_{i=0}^n c_i(z)t^i + o(t^n),$$
where the the coefficients $c_i(z)\; (i=1,2,...,n)$ can be expressed as universal polynomials in the components of the curvature tensor and its covariant derivatives at $z$ (assuming these components are written down in normal coordinates at $z$). Moreover, we describe a method to compute the coefficients $c_i(z)$ for arbitrary $i$ and we compute $c_0(z), c_1(z)$ and $c_2(z)$.

Our main theorem together with the computation of $c_0(z)$ and $c_1(z)$ proves Theorem \ref{entropy}, which essentially states that
\[
D_{KL}(q_Z(t, z, \cdot) \| \vartheta) = - \frac{d}{2} \log(2 \pi t) + \log [\mathrm{Vol}(Z)] - \frac{d}{2} + \frac{1}{4} \mathsf{Sc}(z) t + o(t).
\]
where $\mathsf{Sc}(z)$ is the scalar curvature of the manifold in the point $z \in Z$. We have mentioned Theorem \ref{entropy} in earlier work \cite[Proposition 1]{ijcai2020p375}, but here we follow up with a rigorous proof.

In order to prove these results, we use the so-called parametrix expansion, which is a classical approach to construct the fundamental solution to a partial differential equation equation. Hadamard \cite{hadamard1932probleme} used a parametrix expansion to construct a fundamental solution for the wave equation. Analogous to this construction, Minakshisundaram and Pleijel \cite{minakshisundaram1949some} derived the following parametrix expansion of the heat kernel for a closed Riemannian manifold

\begin{equation}\label{Embeddding in mathematical literature}
	q_Z(t,z,w) \approx \frac{1}{(2\pi t)^{\frac{d}{2}}}\exp\left(-\frac{\dist(w, z)^2}{2t}\right)\left[\sum_{i=0}^{n}u_i(z,w)t^i \right].
\end{equation}
for small $t$ and smooth functions $u_i$ defined on a neighborhood of the diagonal of $Z \times Z$. 
The expansion in Equation (\ref{Embeddding in mathematical literature}) was derived to study the analytic continuation in the s-plane of the Dirichlet series that corresponds to the spectra of the Laplace-Beltrami-Operator.
More generally, the coefficients of a parametrix expansion are Riemannian invariant and encode the relationship between the topology, the geometry and the spectra of a differential operator on the Riemannian manifold  \cite{rosenberg1997laplacian}, \cite{berger1971spectre}.

For our results we need rather strong rigorous error estimates for the approximation (\ref{Embeddding in mathematical literature}). In Theorem \ref{th:kannai}, we record one such estimate that essentially goes back to \cite{kannai1977off}. Similar estimates were obtained by Ben Arous \cite{ben1988developpement} for the more general case of fundamental solution of a hypoelliptic heat equation.

\subsection{Motivation: The Diffusion Variational Autoencoder}
Our mathematical result is needed in a machine-learning algorithm called the Diffusion Variational Autoencoder (\DeltaVAE), which is a variant of a Variational Autoencoder (VAE) that allows for a closed manifold as latent space, rather than the usual Euclidean latent space.

A Variational Autoencoder \cite{kingma2013auto} \cite{kingma2019introduction}  \cite{rezende2014stochastic} is an unsupervised machine-learning algorithm that aims to find a representation of a dataset in a lower-dimensional space. The lower-dimensional space is called the \emph{latent space}. A Variational Autoencoder is built up of two parts, an encoder, which maps data points to points in latent space, and a decoder which maps points in latent space to back to points in data space. Both parts are usually implemented as neural networks and are optimized according to some loss function, which stimulates the composition of encoder and decoder to approximate the identity map.

Because the latent space is usually a relatively low-dimensional latent space, much lower than the data space, a VAE is often thought to compress the data. The heuristic is that if a VAE finds a good compression, then the position of an encoding in latent space contains useful, relevant or even interpretable information about the dataset. Sometimes this is expressed by saying that a VAE manages to \emph{disentangle latent factors} \cite{higgins2016beta,higgins2018towards}. Although it is still unclear what is precisely meant by that, there are attempts to define disentanglement by Higgins et al \cite{higgins2018towards} and to measure disentanglement by P\'erez Rey et al \cite{rey2020metric}.

There are several variants of the VAE that are specifically meant to improve this disentanglement of latent factors. Higgins et al \cite{higgins2016beta} and Burgess et al \cite{burgess2018understanding} introduce the $\beta$-VAE which add an adjustable hyperpamarameter $\beta$ in the loss function of the VAE in order to balance latent channel capacity and independence constraints with reconstruction accuracy. The factorVAE \cite{kim2018disentangling} is a modification of the $\beta$-VAE and disentangles in the way that is forces the distribution of the representation to be factorial in order to have independence in each dimension.

Yet sometimes, a VAE with a standard Euclidean space is fundamentally unable to capture aspects of the dataset that a human observer might view as important latent factors. This is for instance the case in the idealized example when the dataset consists of pictures of a single object, rotated over different angles around one axis. To a human observer, the angle of rotation may be viewed as an important factor. However, when a VAE with one-dimensional latent space is trained on this dataset, either pictures of the object that are very similar end up in very different parts of latent space, or otherwise two nearby sampled points on the latent space can give two far away points or even two meaningless points in the dataspace when decoded \cite{falorsi2018explorations}. This phenomenon was called manifold mismatch by Davidson et al \cite{davidson2018hyperspherical}. 

To overcome the problem of manifold mismatch, one could choose a latent space that is homeomorphic to the topological structure of the data \cite{falorsi2018explorations}. In the case of the example, the latent space could be chosen to be a circle.

Several variants of the VAE have been developed that allow for various topological spaces as latent space. Xu and Durrett \cite{xu2018spherical} and Davidson et al \cite{davidson2018hyperspherical} developed VAEs that allowed for spherical latent spaces, and Falorsi et al \cite{falorsi2018explorations} designed a VAE which allowed for Lie groups as latent space. In order to allow for general closed manifolds as latent space, Pérez Rey et al \cite{ijcai2020p375} developed the Diffusion Variational Autoencoder (\DeltaVAE). There are also approaches that focus on learning a different metric on a space otherwise homeomorphic to $\mathbb{R}^n$, such as the ones by Kalatzis et al \cite{kalatzis2020variational} and by Chadebec
et al \cite{chadebec2020geometry}.

The asymptotic expansion derived in this article is important for the computational efficiency of the \DeltaVAE: The relative entropy is one term in the loss function for the \DeltaVAE, and the asymptotic expansion provides a means to approximate it quickly. This article provides rigorous argumentation for the correctness of this approximation.

\subsection{Organization of the work}
This work is organized as follows. In Section \ref{section2}, we set some notations and review all preliminary results that will be needed in the proof of the main result and the computation of the coefficients $c_i(z)$. In Section \ref{section4}, we prove the existence of an expansion up any order for the entropy, and that the expansion is such that the coefficients are universal polynomials in the component of the curvature tensor and its covariant derivatives at $z$ for $z\in Z$ fixed. Moreover, we express the coefficients $c_i(z)$ in this asymptotic expansion in terms of Gaussian integrals. In Section \ref{lastsection}, we outline a general procedure to compute the Taylor series at $0$ of the square root of the determinant  $\sqrt{\mathrm{det}(g(y))}$ of the Riemannian metric $g$, and the Taylor series at $0$ of the functions $u_i(0,y)$ $(i=0,1,2,\dots,n)$ that appear in the parametrix expansion. We use the lowest order terms of these Taylor series that are already recorded in \cite{sakai1971eigen} to compute the coefficients $c_0(z), c_1(z)$ and $c_2(z)$ of Theorem \ref{Higher-expansion}. The computation of $c_0(z)$ and $c_1(z)$ rigorously recovers the first-order expansion mentioned in \cite[Proposition 1]{ijcai2020p375}.
	\section{Notations and preliminary results}\label{section2}

We first introduce notation for the heat kernel on a Riemannian manifold. We denote by $(Z,g)$ a $d$-dimensional closed and connected Riemannian manifold equipped with a Riemannian metric $g$. We let $q_Z : (0,\infty) \times Z \times Z \to \mathbb{R}$ denote the heat kernel of $Z$, i.e the fundamental solution of the heat equation ~\cite[Equation (1.24)]{rosenberg1997laplacian}

\begin{equation}\label{heatequation}
\begin{cases}
		\partial_t q_Z(t,z,w) = \frac{1}{2}\;\Delta_z q_Z(t,z,w) &\text{for } t\geq 0 \;\mathrm{and }\; z,w\in Z\\
		\lim_{t \downarrow 0} \int q_Z(0,z,w) f(w) d \mathrm{Vol}(w) = f(z)&\text{for } z\in Z \text{ and all } f \in C(Z),
\end{cases}
\end{equation}
where $\mathrm{Vol}$ denotes the standard Riemannian volume measure on $Z$. For fixed $t$ and $z$, the function $w\in Z\longmapsto q_Z(t,z,w)$ is the density with respect to the standard Riemannian volume measure of a probability measure on $Z$.

Next, we will give an expression for the KL-divergence between the kernel $q_Z(t, z, \cdot)$ and the normalized Riemannian volume measure $\vartheta$. The measure $\vartheta$ has density $1/\mathrm{Vol}(Z)$ with respect to the standard Riemannian volume measure $\mathrm{Vol}(\cdot)$.
The KL-divergence can therefore be written as
\begin{equation}\label{entropy_formula}
	\begin{split}
D_{KL}(q_Z(t,z,\cdot)\vert\vert \vartheta)&:= \int_{Z}q_Z(t,z,w)\log \left[\mathrm{Vol}(Z)\;q_Z(t,z,w)\right] d\mathrm{Vol}(w)\\
&= \int_{Z}q_Z(t,z,w)\log \left[\;q_Z(t,z,w)\right] d\mathrm{Vol}(w) + \log[\mathrm{Vol(Z)}]
	\end{split}
\end{equation}
In the article, we are going to give asymptotic expansions of this expression for small values of $t$.


The following lemma allows us to reduce the domain of integration.
\begin{lem}\label{exponential_estimate}
Fix $z\in Z$ and let $\epsilon>0$. Then for any positive integer $n$,
$$\int_Zq_Z(t,z,w)\log[q_Z(t,z,w)]d\mathrm{Vol}(w) = \int_{B_{\epsilon}(z)}q_Z(t,z,w)\log[q_Z(t,z,w)]\;d\mathrm{Vol}(w) +o(t^n),$$
where $$B_{\epsilon}(z):= \{w\in Z:\; \dist(z,w)<\epsilon\}$$ is the ball of radius $\epsilon$ centered at $z$ given by the distance $\dist$ induced by the Riemannian metric $g$.
\end{lem}

\begin{proof}
Let $z$ be an element of $Z$, and let $\epsilon > 0$. Since we consider only small values of $t$, then there exists a constant $C>0$ such that $$q_Z(t,z,w)\leq C t^{-(d+1)}e^{-\frac{\dist(z,w)^2}{4t}}$$
\cite[Theorem 3.5 and Remark 3.6 with $j=m=l=0$ and $n=d$ is the dimension of $Z$]{ludewig2019strong}. It turns out that 
$$q_Z(t,z,w)\leq C t^{-(d+1)}e^{-\frac{\epsilon^2}{4t}}$$ if $\dist(z,w)\geq \epsilon.$

As a result, there exists a positive constant $c=c(z,{\epsilon})$ independent of $w$ such that
\begin{equation}\label{exp_decay_kannai}
0\leq q_Z(t,z,w)\leq e^{-\frac{c}{t}}
\end{equation}
for all $w \in Z \setminus B_{\epsilon}(z)$ and for small values of $t$.

Now, let $w\in Z \setminus B_{\epsilon}(z)$. 
Since the function
$$x\in(0,e^{-1}] \longmapsto x\log(x)$$ is a decreasing function converging to $0$ when $x$ converges to $0$, we have for small and positive values of $t$
 $$-e^{-\frac{c}{t}}\frac{c}{t} \leq q_Z(t,z,w)\log[q_Z(t,z,w)]\leq 0.$$
 
 Since $c$ does not depend on $w$, we have that
 $$0\geq \int_{B_{\epsilon}^c(z)}q_Z(t,z,w)\log[q_Z(t,z,w)]d\mathrm{Vol}(w)\geq  -\;e^{-\frac{c}{t}}\;\frac{c}{t}\;\mathrm{Vol}(Z).$$
Since the volume of $Z$ is finite, then the lower bound $$-\;e^{-\frac{c}{t}}\;\frac{c}{t}\;\mathrm{Vol}(Z)$$ converges to $0$ faster than $t^n$ for any fixed positive integer $n$ when $t$ goes to $0$. Therefore
\begin{align*}
\int_Z q_Z\log(q_Z)d\mathrm{Vol}(w) &= \int_{B_{\epsilon}(z)} q_Z\log(q_Z)d\mathrm{Vol}(w) + \int_{B_{\epsilon}^c(z)} q_Z\log(q_Z)d\mathrm{Vol}(w)\\
&= \int_{B_{\epsilon}(z)} q_Z\log(q_Z)d\mathrm{Vol}(w) + o(t^n).
\end{align*}
\end{proof} 
It is then enough to compute the integral in Equation (\ref{entropy_formula}) on a small neighborhood of $z$ in order to have any order (in $t$)  expansion of the entropy.
\vspace{0.5cm}

Throughout the rest of this work, we fix a point $z\in (Z,g)$ and a ball $U:= B_{\epsilon}(z)$ of radius $\epsilon$ centered at $z$ such that the closure $\bar{U}$ of $U$ does not intersect the cut-locus of $z$ (we also identify $U$ with its image with respect to the exponential map based at $z$). We denote by $y$ the corresponding normal coordinates of $w\in U$, i.e  $\exp^{-1}(w) =  y\in \mathbb{R}^d\cong T_zZ$ and $\dist(z,w)^2 = \vert\vert y\vert\vert^2:= g(y,y)$. Moreover, let $\mathsf{Ric}(z)$ and $\mathsf{Sc}(z):= \mathrm{Trace}(\mathsf{Ric}(z))$ denote respectively  the Ricci curvature  and the scalar curvature of $Z$  at $z$. By Lemma \ref{exponential_estimate}, we have for any positive integer $n$:
$$\int_{Z}q_Z(t,z,w)\log q_Z(t,z,w) d\mathrm{Vol}(w) = \int_{U}q_Z(t,z,w)\log q_Z(t,z,w) d\mathrm{Vol}(w) + o(t^n).$$

We are going to use the following $n$th order ($n\in \{0,1,2,\dots\}$) Minakshisundaram-Pleijel parametrix expansion of $q_Z(t,z,w)$, which is valid for $w$ in the closure $\bar{U}$ of $U$.

\begin{theorem}[Minakshisundaram-Pleijel parametrix expansion]\label{th:kannai}
	Let $K \subset Z$ be compact such that $K \times K$ does not intersect the cut locus. Then there exist smooth functions $u_i: K \times K \to \mathbb{R}$ and $R_i:(0,\infty) \times K \times K \to \mathbb{R}$ such that for all $n\in \mathbb{N}$ and z, $w\in K$, 
\begin{equation}\label{parametrix}
	\begin{split}
		q_Z(t,z,w)		&= \frac{1}{(2\pi t)^{d/2}}\exp\left(-\frac{\vert\vert y\vert\vert^2}{2t}\right)\left[\sum_{i=0}^nu_i(z,w)t^i + R_n(t, z, w)\right],
	\end{split}
\end{equation}
where 
\[
\lim_{t \downarrow 0} \sup_{(z, w) \in K \times K} t^{-n} |R_n(t, z, w)| = 0.
\]
\end{theorem}
This theorem was basically proved by Kannai \cite{kannai1977off}, although his formulation of the theorem was slightly different. 

Furthermore, we have the following common property of the functions $u_i$ $(i\in \{0,1,2,...\})$ and the square root of the determinant of the metric tensor at $w$.
\begin{prop}\label{smoothness}
Let $z\in Z$ and let $R_z,\nabla R_z,\nabla^2R_z,...,\nabla^nR_z,...$ denote the covariant derivatives of the curvature tensor at $z$. For $i\in \mathbb{N}$, let $u_i(0,y)$ be the expression of $u_i(z,w)$ of Equation (\ref{parametrix}) in terms of the normal coordinates (the normal coordinate of $z$ and $w$ are respectively $0$ and $y$). Then the coefficients of the Taylor series at $y=0$ of the function
$$y\in U\longmapsto u_i(0,y)\in \mathbb{R}$$ can be expressed as a  universal polynomials\footnote{These are polynomials that depend only on the dimension of $Z$} in the components of $R_z,\nabla R_z,\nabla^2R_z,...,\nabla^nR_z,...$ expressed in normal coordinates.
The same result holds for the function$$y\in U \longmapsto \sqrt{\mathrm{det}[g(y)]},$$ where $g(y)$ is the matrix of the Riemannian metric at $w$ expressed in terms of the normal coordinate~ $y$.
\end{prop}
\begin{proof}
See Proof of  \cite[Lemma 3.26]{rosenberg1997laplacian}
\end{proof}
Before going to the next section, let us prove some elementary results and some properties of Gaussian integrals.
\begin{defn}\label{Symmetric function}
	Let $f$ be a function defined on a ball $\mathbb{B}$ of center $0$ (and of any radius) of $\mathbb{R}^d$. We say that $f$ is \textit{odd} if $$f(X) = -f(-X)$$ for any $X\in \mathbb{B}$. 
\end{defn}

\begin{lem}\label{simpleLemma}
	Let $f$ be an odd and integrable function on a ball $\mathbb{B}$ of center $0$ of $\mathbb{R}^d$. Then $$\int_\mathbb{B} f(x)dx = 0.$$
\end{lem}

\begin{proof}
	Consider the change of variable $y=-x$. The Jacobian of this change of variables is $1$. Therefore
	\[
	\int_{\mathbb{B}} f(x) dx = \int_{\mathbb{B}} f(-y) dy = - \int_{\mathbb{B}} f(y) dy
	\]
	and we have the result.
\end{proof}

\begin{lem}\label{lemma1}
	Let $\sigma$ be a positive real number and let $y=(y_1,...,y_d)\sim \mathcal{N}(0_d, \sigma\mathbb{I}_d)$ be a $d$-dimensional Gaussian random variable.
	\begin{enumerate}[(i)]
		\item\label{lemma1:i} We have 
		\[
		\int_{\mathbb{R}}e^{-\frac{x^2}{2\sigma}}dx = \sqrt{2\pi \sigma}.
		\]
		\item\label{lemma1:ii} Assume that $\sigma=1$. For $i = 1, \dots, d$, we have
		\begin{align*}
		\int_{\mathbb{R}^d} \frac{1}{(2\pi)^{d/2}}\exp\left(-\frac{\vert\vert  y\vert\vert^2}{2}\right)\vert\vert y\vert\vert^2 (y_i)^2dy& = d+2,\\
			\int_{\mathbb{R}^d} \frac{1}{(2\pi )^{d/2}}\exp\left(-\frac{\vert\vert y\vert\vert^2}{2}\right)\vert\vert y\vert\vert^2 (y_i)^4dy& =3d+12.\\
		\end{align*}
		\item\label{lemma1:iii} Assume that $\sigma=1$. Let $n\in \{0,1,2,...\}$ and let $\epsilon >0$. Let $Y=(Y_1,...,Y_d)\in \mathbb{R}^d$ $(d\in \{1,2,3,...\})$ and let $P(Y)\in \mathbb{R}[Y_1,...,Y_d]$ be a polynomial. For $t>0$, consider
		$$B_{\frac{\epsilon}{\sqrt{t}}}:= \left\{Y\in \mathbb{R}^d: \vert\vert Y\vert\vert=\sqrt{Y_1^2+...+Y_d^2}\leq \frac{\epsilon}{\sqrt{t}}\right\}.$$
		Then we have 
		$$\int_{B_{\frac{\epsilon}{\sqrt{t}}}}\frac{1}{(2\pi)^{d/2}}\exp\left(-\frac{\vert\vert y\vert\vert^2}{2}\right)P(y)dy = \int_{\mathbb{R}^d}\frac{1}{(2\pi)^{d/2}}\exp\left(-\frac{\vert\vert y\vert\vert^2}{2}\right)P(y)dy+ o(t^n)$$ for small values of $t$.
		
		\item\label{lemma1:iv} Let $P(y)=\prod_{i=1}^{d}y_i^{k_i}$ be a monomial of degree $k=k_1+\cdots+k_d$ ($k_1,...,k_d$ positive integers). Then there exists $c\in \mathbb{R}$ independent of $\sigma$ such that $$\int_{\mathbb{R}^d}\frac{1}{(2\pi \sigma)^{d/2}}\exp\left(-\frac{\vert\vert y\vert\vert^2}{2\sigma}\right)P(y)dy= c\sigma^{\frac{k}{2}}.$$
	\end{enumerate}
\end{lem}

\begin{proof}\begin{enumerate}[(i)]
\item   This follows from the probability density function of $\mathcal{N}(0,\sigma)$.
\item   Let $i\in \{1,2,\dots,d\}$.
\begin{enumerate}[1.]
\item Since $\mathbb{E}[y_j^4]=3$, $\mathbb{E}[y_j^2]=1$ $(j\in \{1,2,\dots ,d\})$ and $\vert\vert y\vert\vert^2 = y_1^2+\dots+ y_d^2$, we have
\begin{align*}\displaystyle
	\int_{\mathbb{R}^d} \frac{1}{(2\pi)^{d/2}}\exp\left(-\frac{\vert\vert y\vert\vert^2}{2}\right)\vert\vert y\vert\vert^2 (y_i)^2dy&= \int_{\mathbb{R}} \frac{1}{\sqrt{2\pi}} e^{-\frac{y_i^2}{2}}y_i^4dy_i
	\\ &\quad + \sum_{j\in \{1,\dots,d\}\setminus \{i\}}\left(\int_{\mathbb{R}} \frac{1}{\sqrt{2\pi}} e^{-\frac{y_j^2}{2}}y_j^2dy_j  \right) \left(\int_{\mathbb{R}} \frac{1}{\sqrt{2\pi }} e^{-\frac{y_i^2}{2}}y_i^2dy_i \right)\\
	&= \mathbb{E}[y_i^4] + \mathbb{E}[y_i^2]\;\sum_{j\neq i} \mathbb{E}[y_j^2]\\& = 3 + (d-1)\\
	&= (d+2).
\end{align*}
\item Since $\mathbb{E}[y_j^6]=15$, $\mathbb{E}[y_j^4]=3$ $(j\in \{1,2,\dots ,d\})$ and $\vert\vert y\vert\vert^2 = y_1^2+\dots+ y_d^2$, we have
\begin{align*}\displaystyle
	\int_{\mathbb{R}^d} \frac{1}{(2\pi)^{d/2}}\exp\left(-\frac{\vert\vert y\vert\vert^2}{2}\right)\vert\vert y\vert\vert^2 (y_i)^4dy&= \int_{\mathbb{R}} \frac{1}{\sqrt{2\pi}} e^{-\frac{y_i^2}{2}}y_i^6dy_i 
	\\ &\quad + \sum_{j\in \{1,\dots,d\}\setminus \{i\}}\left(\int_{\mathbb{R}} \frac{1}{\sqrt{2\pi}} e^{-\frac{y_j^2}{2}}y_j^2dy_j  \right) \left(\int_{\mathbb{R}} \frac{1}{\sqrt{2\pi }} e^{-\frac{y_i^2}{2}}y_i^4dy_i \right)\\
	&= \mathbb{E}[y_i^6] +\mathbb{E}[y_i^4]\; \sum_{j\neq i} \mathbb{E}[y_j^2]\\& = 15 + 3(d-1)\\
	&= 3d+12.
\end{align*}
\end{enumerate}
\item   Since any polynomial is a linear combination of monomials, it is enough to assume that $P(Y)= Y_1^{k_1}...Y_d^{k_d}$ with $k_1,...,k_d\in \{0,1,2,...\}.$
\begin{enumerate}[1.]
	\item 
Assume first that $k_i$ is even for $i=1,\dots,d$. Then $P(y)\geq 0$ for all $y\in \mathbb{R}^d$. Consider the cube of dimension $d$ defined by
\begin{align*}
	B_{\frac{\epsilon}{\sqrt{dt}}}^{\infty}&:= \left\{Y=(Y_1,...,Y_d)\in \mathbb{R}^d: \vert\vert Y\vert\vert^{\infty} = \mathsf{max}_i(\vert Y_i\vert)\leq \frac{\epsilon}{\sqrt{dt}} \right\}\\
	&= \left[-\frac{\epsilon}{\sqrt{dt}},\frac{\epsilon}{\sqrt{dt}}\right]^d.
\end{align*}
Since $$\vert \vert Y\vert\vert = \sqrt{Y_1^2+...+Y_d^2} \leq \sqrt{d}\;\mathsf{max}_i (\vert Y_i\vert) =  \sqrt{d}\vert\vert Y\vert\vert^{\infty},$$ we have that $$\displaystyle \left[-\frac{\epsilon}{\sqrt{dt}},\frac{\epsilon}{\sqrt{dt}}\right]^d=B_{\frac{\epsilon}{\sqrt{dt}}}^{\infty}\subseteq B_{\frac{\epsilon}{\sqrt{t}}},$$ i.e.
$$0\leq \int_{[-\frac{\epsilon}{\sqrt{dt}},\frac{\epsilon}{\sqrt{dt}}]^d}\frac{1}{(2\pi)^{d/2}}\exp\left(-\frac{\vert\vert y\vert\vert^2}{2}\right)P(y) dy\leq \int_{B_{\frac{\epsilon}{\sqrt{t}}}}\frac{1}{(2\pi)^{d/2}}\exp\left(-\frac{\vert\vert y\vert\vert^2}{2}\right)P(y) dy.$$
Hence 
\begin{equation}\label{equation_trivial}
	\begin{split}
	0&\leq  \int_{\mathbb{R}^d}\frac{1}{(2\pi)^{d/2}}\exp\left(-\frac{\vert\vert y\vert\vert^2}{2}\right)P(y) dy-\int_{B_{\frac{\epsilon}{\sqrt{t}}}}\frac{1}{(2\pi)^{d/2}}\exp\left(-\frac{\vert\vert y\vert\vert^2}{2}\right)P(y) dy\\
	&\leq \int_{\mathbb{R}^d}\frac{1}{(2\pi)^{d/2}}\exp\left(-\frac{\vert\vert y\vert\vert^2}{2}\right)P(y) dy-\int_{[-\frac{\epsilon}{\sqrt{dt}},\frac{\epsilon}{\sqrt{dt}}]^d}\frac{1}{(2\pi)^{d/2}}\exp\left(-\frac{\vert\vert y\vert\vert^2}{2}\right)P(y) dy\\
	&= \int_{\mathbb{R}^d}\frac{1}{(2\pi)^{d/2}}\exp\left(-\frac{\vert\vert y\vert\vert^2}{2}\right)P(y) dy - \prod_{i=1}^{d}\int_{[-\frac{\epsilon}{\sqrt{dt}},\frac{\epsilon}{\sqrt{dt}}]}\frac{1}{\sqrt{2\pi}}e^{-\frac{ y_i^2}{2}}y_i^{k_i}dy_i,\\
	&= \mathbb{E}[P(y) ] -\prod_{i=1}^{d}\left(\mathbb{E}[y_i^{k_i}]-2 \int_{\frac{\epsilon}{\sqrt{dt}}}^{\infty}\frac{1}{\sqrt{2\pi}}e^{-\frac{y_i^2}{2}}y_i^{k_i}dy_i\right).
\end{split}
\end{equation}

Since $$\displaystyle\mathbb{E}[P(y) ] = \mathbb{E}[ y_1 ^{k_1}... y_d^{k_d}]= \prod_{i=1}^{d}\mathbb{E}[y_i^{k_i}],$$
then the term $\mathbb{E}[P(y)]$ of the upper bound
$$\displaystyle \mathbb{E}[P(y) ] -\prod_{i=1}^{d}\left(\mathbb{E}[y_i^{k_i}]-2 \int_{\frac{\epsilon}{\sqrt{dt}}}^{\infty}\frac{1}{\sqrt{2\pi}}e^{-\frac{y_i^2}{2}}y_i^{k_i}dy_i\right)$$
in Equation (\ref{equation_trivial}) cancels out after expanding the finite product $$\displaystyle \prod_{i=1}^{d}\left(\mathbb{E}[y_i^{k_i}]-2 \int_{\frac{\epsilon}{\sqrt{dt}}}^{\infty}\frac{1}{\sqrt{2\pi}}e^{-\frac{y_i^2}{2}}y_i^{k_i}dy_i\right).$$ 
 Since $\mathbb{E}[y_j ^{k_j}]$ is finite for $j=1,\dots, d$, then it is enough to prove that $\int_{\frac{\epsilon}{\sqrt{dt}}}^{\infty}\frac{1}{\sqrt{2\pi}}e^{-\frac{x^2}{2}}x^{k}dx= o(t^n)$ for any fixed $k\in \{2m:\; m=0,1,2,\dots\}$. 

Indeed, since $\lim\limits_{x\longrightarrow \infty}e^{-\frac{x^2}{2}}x^kx^{4n}=0,$ then we always have $e^{-\frac{x^2}{2}}x^kx^{4n}\leq 1$ for small $t$. i.e. $$e^{-\frac{x^2}{2}}x^k\leq \frac{1}{x^{4n}}$$ for $x\geq \frac{\epsilon}{\sqrt{dt}}.$
Hence, if $n>0$
\begin{align*}
0&\leq \int_{\frac{\epsilon}{\sqrt{dt}}}^{\infty}\frac{1}{\sqrt{2\pi}}e^{-\frac{x^2}{2}}\vert x\vert^{k}dx\\
&\leq \int_{\frac{\epsilon}{\sqrt{dt}}}^{\infty}\frac{1}{x^{4n}}dx = \frac{d^{\frac{4n-1}{2}}\epsilon^{-4n+1}}{4n-1} t^{\frac{4n-1}{2}}\\
&= o(t^n).
\end{align*}
The case $n=0$ follows by bounding $e^{-\frac{x^2}{2}}x^k$ with $\leq \frac{1}{x^{2}}$ (this is always possible with small values of $t$).
This implies that $$\displaystyle \mathbb{E}[P(y) ] -\prod_{i=1}^{d}\left(\mathbb{E}[y_i^{k_i}]-2 \int_{\frac{\epsilon}{\sqrt{dt}}}^{\infty}\frac{1}{\sqrt{2\pi}}e^{-\frac{y_i^2}{2}}y_i^{k_i}dy_i\right) = o(t^n)$$
and the result follows by Equation (\ref{equation_trivial}).
\item Assume now that there exists $i_0\in \{1,\dots,d\}$ such that $k_{i_0}$ is odd. Consider the change of variables $s_0:\mathbb{R}^d\to \mathbb{R}^d$ defined for $y=(y_1,\dots,y_d)\in \mathbb{R}^d$ by 
$$s_0(y) = (y_1,\dots,y_{i_0-1}, -y_{i_0},y_{i_0+1},\dots,y_d).$$ 
Then the Jacobian of $s_0$ is $1$ and we have that
\begin{align*}
\begin{cases}
& \vert \vert y\vert \vert = \vert\vert s_0(y)\vert\vert\\
& P(y) = - P[s_0(y)]
\end{cases}.
\end{align*}
Hence
\begin{align*}
 \int_{\mathbb{R}^d}\frac{1}{(2\pi)^{d/2}}&\exp\left(-\frac{\vert\vert y\vert\vert^2}{2}\right)P(y) dy-\int_{B_{\frac{\epsilon}{\sqrt{t}}}}\frac{1}{(2\pi)^{d/2}}\exp\left(-\frac{\vert\vert y\vert\vert^2}{2}\right)P(y) dy\\
&= \int_{\mathbb{R}^d}\frac{1}{(2\pi)^{d/2}}\exp\left(-\frac{\vert\vert s_0(y)\vert\vert^2}{2}\right)P[s_0(y)] dy-\int_{B_{\frac{\epsilon}{\sqrt{t}}}}\frac{1}{(2\pi)^{d/2}}\exp\left(-\frac{\vert\vert s_0(y)\vert\vert^2}{2}\right)P[s_0(y)] dy\\
&= -\int_{\mathbb{R}^d}\frac{1}{(2\pi)^{d/2}}\exp\left(-\frac{\vert\vert y\vert\vert^2}{2}\right)P(y) dy+\int_{B_{\frac{\epsilon}{\sqrt{t}}}}\frac{1}{(2\pi)^{d/2}}\exp\left(-\frac{\vert\vert y\vert\vert^2}{2}\right)P(y) dy\\
&= -\left(\int_{\mathbb{R}^d}\frac{1}{(2\pi)^{d/2}}\exp\left(-\frac{\vert\vert y\vert\vert^2}{2}\right)P(y) dy-\int_{B_{\frac{\epsilon}{\sqrt{t}}}}\frac{1}{(2\pi)^{d/2}}\exp\left(-\frac{\vert\vert y\vert\vert^2}{2}\right)P(y) dy \right).
\end{align*}
This implies that
$$\displaystyle \int_{\mathbb{R}^d}\frac{1}{(2\pi)^{d/2}}\exp\left(-\frac{\vert\vert y\vert\vert^2}{2}\right)P(y) dy-\int_{B_{\frac{\epsilon}{\sqrt{t}}}}\frac{1}{(2\pi)^{d/2}}\exp\left(-\frac{\vert\vert y\vert\vert^2}{2}\right)P(y) dy=0=o(t^n)$$ and we have the result.
\end{enumerate}

\item   This follows by using the moment of order $p\in \{0,1,2,\dots\}$ of a  Gaussian random variable. More precisely,
\begin{align*}
{\displaystyle \int_{\mathbb{R}} \frac{y_j^p}{(2\pi \sigma)^{d/2}}e^{-\frac{y_j^2}{2\sigma}}dy_j = \mathbb{E}[y_j^p]}=\begin{cases}
	& 0 \;\;\mathrm{if}\;\;p\;\;\mathrm{is} \;\;\mathrm{odd,}\\ & \sigma^{p/2}\; {\displaystyle\prod_{k=0}^{\lceil\frac{p-1}{2}\rceil-1}(p-2k-1)} \;\;\mathrm{if}\;\;p\;\;\mathrm{is}\;\;\mathrm{even}
\end{cases}
\end{align*}
for $j\in \{1,2,...,d\}$, where $\lceil\frac{p-1}{2}\rceil= \mathrm{min}\{m\in \mathbb{N}: m\geq \frac{p-1}{2} \}$.
\end{enumerate}
\end{proof}

As in \cite{sakai1971eigen}, we denote $Y = (Y^1,...,Y^d)\in \mathbb{R}^d$, i.e.~we put the index up for the components of $Y\in \mathbb{R}^d$ and we will use the Einstein summation convention. Furthermore, for $k\in \mathbb{N}$ and for any finite sequence of integer $i_1,...,i_k\in \{1,...,d\}$, we denote:
\begin{equation}\label{MomentGaussianNotation}
	\begin{split}
		\mu^{i_1...i_k}&:= \int_{\mathbb{R}^d}\frac{1}{(2\pi)^{d/2}}\exp\left(-\frac{\vert\vert y\vert\vert^2}{2}\right)Y^{i_1}...Y^{i_k}\frac{\vert\vert Y\vert\vert^2}{2}dY,\\
		\nu^{i_1...i_k}&:= \int_{\mathbb{R}^d}\frac{1}{(2\pi)^{d/2}}\exp\left(-\frac{\vert\vert y\vert\vert^2}{2}\right)Y^{i_1}...Y^{i_k}dY.
	\end{split}
\end{equation}
We will need the following contraction Lemma.
\begin{lem}\label{contractionlemma}
	Let $d\in \{1,2,...\}$. For $i,j,k,l\in \{1,...,d\}$, let $\mathsf{R}_{ijkl}$ and $T_{ij}$ be real numbers. Then 
	\begin{enumerate}[(i)]
		\item $\mathsf{R}_{ijkl}\;\mu^{ijkl}=\frac{d+4}{2}\left(\mathsf{R}_{iijj}+\mathsf{R}_{ijij}+\mathsf{R}_{ijji} \right)$,
		\item $\mathsf{R}_{ijkl}\;\nu^{ijkl}=\mathsf{R}_{iijj}+\mathsf{R}_{ijij}+ \mathsf{R}_{ijji}$,
		\item $T_{ij}\;\mu^{ij}= \frac{d+2}{2}\;\mu_{ii}=\frac{d+2}{2}\;\mathrm{Trace}(T)$,
		\item $T_{ij}\;\nu^{ij}= \nu_{ii}= \mathrm{Trace}(T)$
	\end{enumerate}
\end{lem}
\begin{proof}
	$(i)\;$ By using Lemma \ref{lemma1} (\ref{lemma1:ii}), we have
	\begin{align*}\displaystyle
		\mu^{ijkl}&= \int_{\mathbb{R}^d} \frac{1}{(2\pi)^{d/2}} \exp\left(-\frac{\vert\vert y\vert\vert^2}{2}\right)\frac{\vert\vert Y\vert\vert^2}{2}Y^iY^jY^kY^l dY\\
		&=\begin{cases}
			0&\mathrm{if}\; \#\{i,j,k,l\}=4 \;\mathrm{or}\;\#\{i,j,k,l\}=3\;\\&\mathrm{or}\;\left(\#\{i,j,k,l\}=2\;\mathrm{and}\;\mathrm{exactly}\;\mathrm{three}\; \mathrm{of}\;i,j,k\;\mathrm{and}\;l\;\mathrm{are\; equal}\right)\\
			\frac{d+4}{2}&\mathrm{if}\;(i=j,\;k=l,\;i\neq l)\;\mathrm{or}\;(i=k,\;j=l,
			i\neq l)\;\mathrm{or}\;(i=l,\;j=k,\;i\neq k)\\
			\frac{3d+12}{2}&\mathrm{if} \;i=j=k=l.
		\end{cases}\\
	\end{align*}
	Hence
	\begin{align*}
		\mathsf{R}_{ijkl}\;\mu^{ijkl}&= \frac{3d+12}{2} \sum_{i=1}^d\mathsf{R}_{iiii} +\frac{d+4}{2}\; \sum_{i\neq j}\left(\mathsf{R}_{iijj}+\mathsf{R}_{ijij}+\mathsf{R}_{ijji} \right)\\
		&= \frac{3d+12}{2} \sum_{i=1}^d\mathsf{R}_{iiii}-3\;\frac{d+4}{2} \sum_{i=1}^d\mathsf{R}_{iiii} + \frac{d+4}{2}\; \sum_{i,j}\left(\mathsf{R}_{iijj}+\mathsf{R}_{ijij}+\mathsf{R}_{ijji} \right)\\
		&= \frac{d+4}{2}\sum_{i,j}\left(\mathsf{R}_{iijj}+\mathsf{R}_{ijij}+\mathsf{R}_{ijji} \right)\\
		&=\frac{d+4}{2}\left(\mathsf{R}_{iijj}+\mathsf{R}_{ijij}+\mathsf{R}_{ijji} \right),
	\end{align*}
	where the last equality is just the Einstein summation convention.\\
	$(ii)$ By using the moments of standard Gaussian random variables, we have
	\begin{align*}
		\nu^{ijkl}&= \int_{\mathbb{R}^d} \frac{1}{(2\pi)^{d/2}} \exp\left(-\frac{\vert\vert y\vert\vert^2}{2}\right)Y^iY^jY^kY^l dY\\
		&=\begin{cases}
			0&\mathrm{if}\; \#\{i,j,k,l\}=4 \;\mathrm{or}\;\#\{i,j,k,l\}=3\\&\mathrm{or}\;\left(\#\{i,j,k,l\}=2\;\mathrm{and}\;\mathrm{exactly}\;\mathrm{three}\; \mathrm{of}\;i,j,k\;\mathrm{and}\;l\;\mathrm{are\; equal}\right)\\
			1&\mathrm{if}\;(i=j,\;k=l,\;i\neq l)\;\mathrm{or}\;(i=k,\;j=l,
			i\neq l)\;\mathrm{or}\;(i=l,\;j=k,\;i\neq k)\\
			3&\mathrm{if} \;i=j=k=l.
		\end{cases}\\
	\end{align*}
	Then
	\begin{align*}
		R_{ijkl}\;\nu^{ijkl}&= 3\;\sum_iR_{iiii} + \sum_{i\neq j}\left(R_{iijj}+R_{ijij}+ R_{ijji} \right)\\&= 3\;\sum_iR_{iiii}-3\;\sum_iR_{iiii} + \sum_{i,j}\left(R_{iijj}+R_{ijij}+ R_{ijji} \right)\\
		&= R_{iijj}+R_{ijij}+ R_{ijji},
	\end{align*}
	where the last equality is just the Einstein summation convention.\\
	$(iii)$ By using Lemma \ref{lemma1} (\ref{lemma1:ii}), we have
	\begin{align*}\displaystyle
		\mu^{ij}&= \int_{\mathbb{R}^d} \frac{1}{(2\pi)^{d/2}} \exp\left(-\frac{\vert\vert Y\vert\vert^2}{2}\right)\frac{\vert\vert Y\vert\vert^2}{2}Y^iY^j dY\\
		&=\begin{cases}
			0\;\;\;\mathrm{if}\; i\neq j\\
			\frac{d+2}{2}\;\;\;\mathrm{if} \;i=j.
		\end{cases}\\
	\end{align*}
	Hence
	\begin{align*}
		T_{ij}\;\mu^{ij}&= \frac{d+2}{2}\;\mathrm{Trace}(T). 
	\end{align*}
	
	$(iv)$ By using the moment of standard Gaussian random variables, we have
	\begin{align*}\displaystyle
		\nu^{ij}&= \int_{\mathbb{R}^d} \frac{1}{(2\pi)^{d/2}} \exp\left(-\frac{\vert\vert Y\vert\vert^2}{2}\right)Y^iY^j dY\\
		&=\begin{cases}
			0\;\;\;\mathrm{if}\; i\neq j\\
			1\;\;\;\mathrm{if} \;i=j.
		\end{cases}\\
	\end{align*}
	Hence
	\begin{align*}
		T_{ij}\;\nu^{ij}&= \mathrm{Trace}(T). 
	\end{align*}
	
\end{proof}

	\section{Existence and property  of an expansion of any order}\label{section4}
In this section we are going to prove, for any given positive integer $n$, the existence of a higher order (in $t$) expansion of the form: 
\begin{equation}
\int_{Z}q_Z(t,z,w)\log q_Z(t,z,w) \;d\mathrm{Vol}(w) = -\frac{d}{2}\log(2\pi t) + \sum_{i=0}^{n}c_i(z)t^i + o(t^n),
\end{equation}
where $c_i(z)$'s $(i=0,1,...,n)$ are universal polynomials in the components of the curvature  tensor and its covariant derivatives at $z$.
The main result can be summarized in the following theorem.
\begin{theorem}\label{Higher-expansion}
	Let $z\in (Z,g)$ and let $n\geq 0$. Then
\begin{enumerate}[(1)]
\item \label{thm:part1}  there exist $c_0(z),c_1(z),...,c_n(z) \in \mathbb{R}$ such that
	$$D_{KL}\left(q_Z(t,z,.) \vert\vert \vartheta\right) = -\frac{d}{2}\log(2\pi t) + \log[\mathrm{Vol}(Z)] +\sum_{i=0}^n c_n(z)t^n + o(t^n).$$
	Furthermore, the coefficient $c_i(z)$ $(i=0,1,..,n)$  can be expressed as universal polynomials in the components of the curvature tensor and its covariant derivatives at $z$, when these are expressed in normal coordinates.
\item \label{thm:part2} More precisely, the coefficient $c_i(z)$  $(i = 1, 2, \dots , n)$ is given by the Gaussian integral
	\[
	c_i(z) = \int_{\mathbb{R}^d}\frac{1}{(2\pi)^{\frac{d}{2}}}\exp\left(-\frac{\vert\vert Y\vert\vert^2}{2}\right)\left(-P_i(Y)\frac{\vert\vert Y\vert\vert^2}{2} + Q_i(Y)\right)dY
	\]
	where $P_i(Y)$ and $Q_i(Y)$ are polynomials in $Y$. The coefficients of $P_i(Y)$ and $Q_i(Y)$ can be expressed as universal polynomials in the components of the curvature tensor and its covariant  derivatives at $z$. The polynomials $P_i(Y)$ and $Q_i(Y)$ are defined as follow.

Let $u_i(0,y)$ $(i=0,...,n)$  be the $i$th order parametrix of $Z$ (cf. Equation (\ref{parametrix}) expressed in terms of the normal coordinates $y$ at $z$. Then the two polynomials $P_i(Y)$ and $Q_i(Y)$ $(i=0,1,...,n)$  can be computed by the following two steps.
\begin{enumerate}[Step 1.]\label{Methode}
	\item \label{step1} Sum up all the terms $T$ of the Taylor series at $(t=0,y=0)$ of the function
	$$\displaystyle F: (t,y)\in \mathbb{R}\times B_{\epsilon}(0)\longmapsto \sqrt{\mathrm{det}[g(y)]}\;[u_0(0,y)+tu_1(0,y) + \dots + t^nu_n(0,y)] $$ such that 
	\begin{itemize}
		\item $T$ is of degree $p$ in $t$ and $T$ is of degree $q$ in the components of $y$ (with $q$ even),
		\item $p$ and $q$ are integers such that $i=p+\frac{q}{2},$
	\end{itemize}
	and let $F_i(t,y)$ be the obtained quantity.\\
	Do the same procedure for the function $$\displaystyle G: (t,y)\in \mathbb{R}\times B_{\epsilon}(0)\longmapsto F(t,y)\;\log[u_0(0,y)+tu_1(0,y) + \dots + t^nu_n(0,y)],$$
	and let $G_i(t,y)$ be the obtained quantity.

	\item \label{step2} Replace $y$ with $\sqrt{t}Y$ in the expressions of $F_i(t,y)$ and $G_i(t,y)$  so that we obtain two polynomials $P_i(Y),Q_i(Y)\in \mathbb{R}[Y]$ satisfying
	\begin{align*}
		t^iP_i(Y) &= F_i(t,\sqrt{t}Y)\\
		t^iQ_i(Y)&= G_i(t,\sqrt{t}Y).
	\end{align*}	
\end{enumerate}

\end{enumerate}
\end{theorem}
\begin{proof}
We are going to prove (\ref{thm:part1}) and (\ref{thm:part2}) of Theorem \ref{Higher-expansion} at the same time.
Recall that from Lemma \ref{exponential_estimate}, it is enough to compute the entropy in a normal neighborhood $U:= B_{\epsilon}(z)$ of $z$, so that we can use the parametrix expansion in (\ref{parametrix}).

By using the normal coordinates $y\in \mathbb{R}^d$ that correspond to the normal neighborhood $U$ of $z$, we have the integral on $U$ of $q_Z(t,z,w)\log q_Z(t,z,w)$ with respect to $d\mathrm{Vol}(w)$
\begin{align*}
\int_U & q_Z(t, z, w)\log q_Z(t, z, w)d\mathrm{Vol}(w)\\
&=  \int_U\frac{1}{(2\pi t)^{\frac{d}{2}}}\exp\left(-\frac{\vert\vert y\vert\vert^2}{2t}\right)\left(\sum_{i=0}^nu_i(0,y)t^i +o(t^n)\right)\sqrt{\mathrm{det}(g)} \log q_Z(t,0,y)dy\\
&=-\frac{d}{2}\log(2\pi t)\int_U\frac{1}{(2\pi t)^{\frac{d}{2}}}\exp\left(-\frac{\vert\vert y\vert\vert^2}{2t}\right)\left(\sum_{i=0}^nu_i(0,y)t^i +o(t^n)\right)\sqrt{\mathrm{det}(g(y))}dy\\
&\quad -  \int_U\frac{1}{(2\pi t)^{\frac{d}{2}}}\exp\left(-\frac{\vert\vert y\vert\vert^2}{2t}\right)\left(\sum_{i=0}^nu_i(0,y)t^i \right)\sqrt{\mathrm{det}(g(y))}  \frac{\vert\vert y\vert\vert^2}{2t}dy\\
&\quad + \int_U\frac{1}{(2\pi t)^{\frac{d}{2}}}\exp\left(-\frac{\vert\vert y\vert\vert^2}{2t}\right)\left(\sum_{i=0}^nu_i(0,y)t^i \right)\sqrt{\mathrm{det}(g(y))} \log \left(\sum_{i=0}^nu_i(0,y)t^i \right)dy\\
&\quad + \int_U\frac{1}{(2\pi t)^{\frac{d}{2}}}\exp\left(-\frac{\vert\vert y\vert\vert^2}{2t}\right)\left(\sum_{i=0}^nu_i(0,y)(0,y)t^i \right)\sqrt{\mathrm{det}(g(y))} \log \left[1+\frac{o(t^n)}{\sum_{i=0}^nu_i(0,y)t^i}\right]dy
\end{align*}
By the fact that $\int_Zq_z(t,z,w)d\mathrm{Vol}(w)=1$ and by the estimate in Equation (\ref{exp_decay_kannai}), we have
\begin{align*}
  \frac{d}{2}\log(2\pi t)&\int_U\frac{1}{(2\pi t)^{\frac{d}{2}}}\exp\left(-\frac{\vert\vert y\vert\vert^2}{2t}\right)\left(\sum_{i=0}^nu_i(0,y)t^i +o(t^n)\right)\sqrt{\mathrm{det}(g(y))}dy \\&= \frac{d}{2}\log(2\pi t) \int_Uq_z(t,z,w)d\mathrm{Vol}(w)\\
  &= \frac{d}{2}\log(2\pi t) + o(t^n).
\end{align*}
Moreover, the quantity $$\log \left[1+\frac{o(t^n)}{\sum_{i=0}^nu_i(0,y)t^i}\right]$$ is well defined since $u_0(0,0)=1$, so that $$\sum_{i=0}^nu_i(0,y)t^i\geq \frac{1}{2}$$
for small values of $t$.

Therefore, we have the following equality
\begin{equation}\label{2ndEualityt}
	\begin{split}
\int_Uq_Z(t,z,w)\log q_Z(t,z,w)d\mathrm{Vol}(w)&=-\frac{d}{2}\log(2\pi t) - K_1(t)+K_2(t)+K_3(t) + o(t^n),
    \end{split}
\end{equation}
where
\begin{equation}\label{K_1}
K_1(t) = \int_U\frac{1}{(2\pi t)^{\frac{d}{2}}}\exp\left(-\frac{\vert\vert y\vert\vert^2}{2t}\right)\left(\sum_{i=0}^nu_i(0, y) t^i \right)\sqrt{\mathrm{det}(g(y))}  \frac{\vert\vert y\vert\vert^2}{2t}dy,
\end{equation}
\begin{equation}\label{K_2}
K_2(t) = \int_U\frac{1}{(2\pi t)^{\frac{d}{2}}}\exp\left(-\frac{\vert\vert y\vert\vert^2}{2t}\right)\left(\sum_{i=0}^nu_i(0, y)t^i \right)\sqrt{\mathrm{det}(g(y))} \log \left(\sum_{i=0}^nu_i(0, y) t^i \right)dy,
\end{equation}
and

\begin{equation}\label{K_3}
K_3(t) = \int_U\frac{1}{(2\pi t)^{\frac{d}{2}}}\exp\left(-\frac{\vert\vert y\vert\vert^2}{2t}\right)\left(\sum_{i=0}^nu_i(0, y)t^i \right)\sqrt{\mathrm{det}(g(y))} \log \left[1+\frac{o(t^n)}{\sum_{i=0}^nu_i(0, y)t^i}\right]dy.
\end{equation}
Throughout this proof, we use the same notation $U$ or $B_{\epsilon}$ for both the neighborhood of $z$ in $Z$ and the neighborhood of $0$ in $\mathbb{R}^d$.  	
	
In order to compute the terms $K_1$, $K_2$ and $K_3$ of Equation (\ref{2ndEualityt}), we are going use the following change of variables

\begin{equation}\label{change_of_variable}
y = \sqrt{t}Y.
\end{equation}
Since $y\in \mathbb{R}^d$, then the Jacobian of this change of variables is just $t^{\frac{d}{2}}$. In addition, we have $Y\in B_{\frac{\epsilon}{\sqrt{t}}}$ because $y\in B_{\epsilon}$.

\begin{enumerate}[(i)]
\item \underline{\textbf{Computation of $K_1(t)$ (\ref{K_1}) and $P_i(Y)$ $(i=0,1,2,...,n)$}}

Consider the function $F: \mathbb{R} \times U\to \mathbb{R}$ defined for $(t,y)\in \mathbb{R}\times U$ by

\begin{equation}\label{functionF}
	F(t,y):= \sqrt{\mathrm{det}[g(y)]}\;\sum_{i=0}^nu_i(0,y)t^i,
\end{equation}
which is $\mathcal{C}^{\infty}$ around $(t,y)=(0,0)$ because $g$ is smooth and non-singular on $U$ and the $u_i$'s are $\mathcal{C}^{\infty}$. Then $F$ has the Taylor expansion at $(0,0)$
  \[F(t,y) = \sum_{i=0}^N \frac{F^{(i)}_{(0,0)} \left[(t,y),...,(t,y)\right]}{i!} + o\left(\mathrm{max}(t^N,\vert\vert y\vert\vert^N)\right),\]
   where $F^{(i)}_{(0,0)}$ is the $i$th derivative of $F$ at $(0,0)$ for $i=0,1,...,N$. If we plug this Taylor expansion into Equation (\ref{K_1}), we have
\begin{align*}
K_1 (t) &= \int_{B_{\epsilon}}\frac{1}{(2\pi t)^\frac{d}{2}}\exp\left(-\frac{\vert\vert y\vert\vert^2}{2t}\right)\sum_{i=0}^N \frac{F^{(i)}_{(0,0)} \left[(t,y),...,(t,y)\right]}{i!}\;\frac{\vert\vert y\vert\vert^2}{2t}dy \\& \quad  + \int_{B_{\epsilon}}\frac{1}{(2\pi t)^\frac{d}{2}}\exp\left(-\frac{\vert\vert y\vert\vert^2}{2t}\right) o\left(\mathrm{max}(t^N,\vert\vert y\vert\vert^N)\right) \frac{\vert\vert y\vert\vert^2}{2t}dy
 \end{align*}

 and by using the change of variables in (\ref{change_of_variable}), we have that
 \begin{align*}
 K_1 (t) &= \int_{B_{\frac{\epsilon}{\sqrt{t}}}}\frac{1}{(2\pi )^\frac{d}{2}}\exp\left(-\frac{\vert\vert Y\vert\vert^2}{2}\right)\sum_{i=0}^N \frac{F^{(i)}_{(0,0)} \left[(t,\sqrt{t}Y),...,(t,\sqrt{t}Y)\right]}{i!}\;\frac{\vert\vert Y\vert\vert^2}{2}dY\\&\quad + \int_{B_{\frac{\epsilon}{\sqrt{t}}}}\frac{1}{(2\pi )^\frac{d}{2}}\exp\left(-\frac{\vert\vert Y\vert\vert^2}{2}\right) o\left(\mathrm{max}(t^N,\vert\vert \sqrt{t}Y\vert\vert^N)\right) \frac{\vert\vert Y\vert\vert^2}{2}dY.
 \end{align*}

For $Y \in B_{\frac{\epsilon}{\sqrt{t}}}$, we have
\begin{align*}
\mathrm{max}(t^N,\vert\vert \sqrt{t}Y\vert\vert^N)=\begin{cases}
\displaystyle t^N & \mathrm{if}\;\; 0\leq \vert\vert Y\vert\vert \leq \sqrt{t}\\
\displaystyle\sqrt{t}^N\vert\vert Y\vert\vert^N &\mathrm{if}\; \sqrt{t}\leq \vert\vert Y\vert\vert\leq  \epsilon/\sqrt{t}.
\end{cases}
\end{align*}
Hence, by taking $N$ large enough, we have
  $$\int_{B_{\frac{\epsilon}{\sqrt{t}}}}\frac{1}{(2\pi )^\frac{d}{2}}\exp\left(-\frac{\vert\vert Y\vert\vert^2}{2}\right) o\left(\mathrm{max}(t^N,\vert\vert \sqrt{t}Y\vert\vert^N)\right) \frac{\vert\vert Y\vert\vert^2}{2}dY=o(t^n)$$
since the moment of any order of a Gaussian random variable is finite.

Therefore
\begin{equation}\label{K_1_simple}
	K_1(t) = \int_{B_{\frac{\epsilon}{\sqrt{t}}}}\frac{1}{(2\pi )^\frac{d}{2}}\exp\left(-\frac{\vert\vert Y\vert\vert^2}{2}\right)\sum_{i=0}^N \frac{F^{(i)}_{(0,0)} \left[(t,\sqrt{t}Y),...,(t,\sqrt{t}Y)\right]}{i!}\;\frac{\vert\vert Y\vert\vert^2}{2}dY + o(t^n).
\end{equation}
Since the derivatives $F^{(i)}_{(0,0)}$ ($i=0,...,N$) are multi-linear forms and by Lemma \ref{lemma1} (\ref{lemma1:iii}), we can compute the integral of Equation (\ref{K_1_simple}) over $\mathbb{R}^d$ for small values of $t$ , i.e.

\begin{equation}\label{K_1finalvalue}
	K_1(t) = \int_{\mathbb{R}^d}\frac{1}{(2\pi )^\frac{d}{2}}\exp\left(-\frac{\vert\vert Y\vert\vert^2}{2}\right)\sum_{i=0}^N \frac{F^{(i)}_{(0,0)} \left[(t,\sqrt{t}Y),...,(t,\sqrt{t}Y)\right]}{i!}\;\frac{\vert\vert Y\vert\vert^2}{2}dY + o(t^n).
\end{equation}
Again, since the derivatives are multi-linear forms, the quantity $$\sum_{i=0}^N \frac{F^{(i)}_{(0,0)} \left[(t,\sqrt{t}Y),...,(t,\sqrt{t}Y)\right]}{i!}$$ is a finite sum of monomials of the form 
$$c_{k,\alpha}(z)t^k(\sqrt{t}Y)^{\alpha},$$
where $\alpha$ is a multi-index, i.e a tuple of positive integers, $c_{k,\alpha} := c_{k,\alpha}(z)\in \mathbb{R}$ and $$(\sqrt{t}Y)^{\alpha}= (\sqrt{t}Y_1)^{\alpha_1}...(\sqrt{t}Y_d)^{\alpha_d}.$$ By Proposition \ref{smoothness}, the coefficients of the Taylor series at $0$ of the functions $y\in U\longmapsto \sqrt{\mathrm{det}[g(y)]}$ and $(t,y)\in \mathbb{R}\times U\longmapsto\sum_{i=0}^{n} u_i(0,y)t^i$ are universal polynomials in the components of the curvature tensor and its covariant derivatives at $z$. Consequently, the coefficients $c_{k,\alpha}$ are also universal polynomials in the components of the curvature tensor and its covariant derivatives at $z$.

Assume now that $\alpha_1+...+\alpha_d$ is odd so that we have a term with a non-integer power of $t$. In this case, there exists at least an index $i_0\in \{0,1,...,d\}$ such that $\alpha_{i_0}$ is odd and the number of such $i_0$'s must also be odd. Hence, for a fixed $t$, the function $$Y\in B_{\frac{\epsilon}{\sqrt{t}}}\longmapsto\frac{1}{(2\pi)^{\frac{d}{2}}}\exp\left(-\frac{\vert\vert Y\vert\vert^2}{2}\right)c_{k,\alpha}t^k(\sqrt{t}Y)^{\alpha}$$ is an odd function in $Y$ on the ball $B_\frac{\epsilon}{\sqrt{t}}$ and its integral on that ball is then $0$ by Lemma \ref{simpleLemma}.

Now, for $i=0,1,2,...,n$ let $F_i(t,\sqrt{t}Y)$ be the sum of all terms $T$ of the quantity $$\displaystyle \sum_{i=0}^N \frac{F^{(i)}_{(0,0)} \left[(t,\sqrt{t}Y),...,(t,\sqrt{t}Y)\right]}{i!}$$ that satisfy the following two condition:
\begin{itemize}
	\item $T$ is of degree $p\in \{0,1,2,...\}$ in $t$ and $T$ is of degree $q\in \{2k:\;k=0,1,2,...\}$ in the components of $\sqrt{t}Y = (\sqrt{t}Y_1,...,\sqrt{t}Y_d)$,
	\item $p$ and $q$ are such that $i = p+\frac{q}{2}$.
\end{itemize}
Then there exist a polynomial $P_i(Y)\in \mathbb{R}[Y]$ such that $$F_i(t,\sqrt{t}Y) = t^i\;P_i(Y),$$
 and we have
\begin{equation}\label{K_1_finalvalue}
\begin{split}
K_1(t) &= \int_{\mathbb{R}^d}\frac{1}{(2\pi )^\frac{d}{2}}\exp\left(-\frac{\vert\vert Y\vert\vert^2}{2}\right)\sum_{i=0}^N \frac{F^{(i)}_{(0,0)} \left[(t,\sqrt{t}Y),...,(t,\sqrt{t}Y)\right]}{i!}\;\frac{\vert\vert Y\vert\vert^2}{2}dY +o(t^n)\\
&=\sum_{i=0}^nt^i\; \int_{\mathbb{R}^d}\frac{1}{(2\pi )^\frac{d}{2}}\exp\left(-\frac{\vert\vert Y\vert\vert^2}{2}\right)P_i(Y)\;\frac{\vert\vert Y\vert\vert^2}{2}dY+ o(t^n).
\end{split}
\end{equation}
 Furthermore, the coefficients of $P_i(Y)$ $(i=0,1,...,d)$ can be expressed as universal polynomials in the components of the curvature tensor and its covariant derivatives at $z$ (cf. Proposition \ref{smoothness}).

\item \underline{\textbf{Computation of $K_2(t)$ (\ref{K_2}) and $Q_i(Y)$ $(i=0,1,2,...,n)$}}

We use exactly the same methods as in the computation of $K_1(t)$. Consider the function $G: \mathbb{R}\times U \to \mathbb{R}$ defined for $(t,y)\in \mathbb{R}\times U$ by
\begin{equation}\label{functionG}
\begin{split}
	G(t,y)&:= \left(\sum_{i=0}^nu_i(0,y)t^i \right)\sqrt{\mathrm{det}(g(y))} \log \left(\sum_{i=0}^nu_i(0,y)t^i \right)\\
	&=F(t,y)\log \left(\sum_{i=0}^nu_i(0,y)t^i \right),
\end{split}
\end{equation}
where $F(t,y)$ is defined in Equation  (\ref{functionF}).

The function $G$ is indeed $\mathcal{C}^{\infty}$ at $(t,y)=(0,0)$ because $(t,y)\in \mathbb{R}\times B_{\epsilon}\longmapsto\sum_{i=0}^nu_i(0,y)t^i$ is smooth, $u_0(0,0)=1$ and $g$ is smooth and non-singular.

Then at $(0,0)$, the function $G$ has the Taylor expansion
$$G(t,y) = \sum_{i=0}^N \frac{G^{(i)}_{(0,0)} \left[(t,y),...,(t,y)\right]}{i!} + o\left(\mathrm{max}(t^N,\vert\vert y\vert\vert^N)\right),$$ where $G^{(i)}_{(0,0)}$ is the $i$th derivative of $G$ at $(0,0)$  for $i=0,...,N$. If we plug this Taylor expansion into Equation (\ref{K_2}), we have
\begin{align*}K_2(t)
	&= \int_{B_{\epsilon}}\frac{1}{(2\pi t)^\frac{d}{2}}\exp\left(-\frac{\vert\vert y\vert\vert^2}{2t}\right)\sum_{i=0}^N \frac{G^{(i)}_{(0,0)} \left[(t,y),...,(t,y)\right]}{i!}\;dy \\
	&\quad+ \int_{B_{\epsilon}}\frac{1}{(2\pi t)^\frac{d}{2}}\exp\left(-\frac{\vert\vert y\vert\vert^2}{2t}\right) o\left(\mathrm{max}(t^N,\vert\vert y\vert\vert^N)\right) dy;
\end{align*}

and by using the change of variables in (\ref{change_of_variable}), we have that
\begin{align*}
	K_2(t) &= \int_{B_{\frac{\epsilon}{\sqrt{t}}}}\frac{1}{(2\pi )^\frac{d}{2}}\exp\left(-\frac{\vert\vert Y\vert\vert^2}{2}\right)\sum_{i=0}^N \frac{G^{(i)}_{(0,0)} \left[(t,\sqrt{t}Y),...,(t,\sqrt{t}Y)\right]}{i!}\;\frac{\vert\vert Y\vert\vert^2}{2}dY\\&\quad + \int_{B_{\frac{\epsilon}{\sqrt{t}}}}\frac{1}{(2\pi )^\frac{d}{2}}\exp\left(-\frac{\vert\vert Y\vert\vert^2}{2}\right) o\left(\mathrm{max}(t^N,\vert\vert \sqrt{t}Y\vert\vert^N)\right) dY.
\end{align*}
Again, since the moment of any order of a Gaussian vector is finite, we have $$\int_{B_{\frac{\epsilon}{\sqrt{t}}}}\frac{1}{(2\pi )^\frac{d}{2}}\exp\left(-\frac{\vert\vert Y\vert\vert^2}{2}\right) o\left(\mathrm{max}(t^N,\vert\vert \sqrt{t}Y\vert\vert^N)\right) dY=o(t^n)$$
by taking $N$ large enough.

Therefore
\begin{equation}\label{K_2_simple}
	K_2(t) = \int_{B_{\frac{\epsilon}{\sqrt{t}}}}\frac{1}{(2\pi )^\frac{d}{2}}\exp\left(-\frac{\vert\vert Y\vert\vert^2}{2}\right)\sum_{i=0}^N \frac{G^{(i)}_{(0,0)} \left[(t,\sqrt{t}Y),...,(t,\sqrt{t}Y)\right]}{i!}\;dY + o(t^n).
\end{equation}
Since the derivatives are multi-linear forms, then by  Lemma \ref{lemma1} (\ref{lemma1:iii}), we can integrate the second member of Equation (\ref{K_2_simple}) over $\mathbb{R}^d$, i.e. 
\begin{equation}\label{K_2finalvalue}
	K_2(t) = \int_{\mathbb{R}^d}\frac{1}{(2\pi )^\frac{d}{2}}\exp\left(-\frac{\vert\vert Y\vert\vert^2}{2}\right)\sum_{i=0}^N \frac{G^{(i)}_{(0,0)} \left[(t,\sqrt{t}Y),...,(t,\sqrt{t}Y)\right]}{i!}\;dY + o(t^n).
\end{equation}
Moreover, the quantity $$\sum_{i=0}^N \frac{G^{(i)}_{(0,0)} \left[(t,\sqrt{t}Y),...,(t,\sqrt{t}Y)\right]}{i!}$$ is a finite sum of monomials of the form 
$$c_{k,\alpha}(z)t^k(\sqrt{t}Y)^{\alpha},$$
where $c_{k,\alpha}:=c_{k,\alpha}(z)\in \mathbb{R}$, $\alpha = (\alpha_1,...,\alpha_d)$ is a multi-index, i.e a tuple of positive integers and $$(\sqrt{t}Y)^{\alpha}= (\sqrt{t}Y_1)^{\alpha_1}...(\sqrt{t}Y_d)^{\alpha_d}.$$

Since $u_0(0,0)=1$, the Taylor series at $(0,0)$ of the function $(t,y)\in \mathbb{R}\times U\longmapsto \sum_{i=0}^{n}u_i(0,y)t^i$ is of the form
$$ 1 + \sum_{i=1}^{\infty} P_i(t,y),$$ where $P_i$ is homogeneous polynomial of degree $i$ in $t,y_1,...,y_d$ for $i>0$ and the coefficients of $P_i$ ($i\in \mathbb{N}$) are universal polynomials in the components of the curvature tensor and its covariant derivatives at $z$ (by Proposition \ref{smoothness}).  Hence, the coefficients of the Taylor series at $(0,0)$ of the function $$(t,y)\in \mathbb{R}\times U\longmapsto \log\left(\sum_{i=0}^{n}u_i(0,y)t^i\right)$$ are also universal polynomial on the components of the curvature tensor and its covariant derivatives at $z$. Hence, by Proposition \ref{smoothness}, all the coefficients $c_{k,\alpha}(z)$ are universal polynomials on the component of the curvature and its covariant derivatives at $z$.

Assume now that $\alpha_1+...+\alpha_d$ is odd so that we have a term with  non-integer power of $t$ in the integral (\ref{K_1_simple}). In this case, there exist at least an index $i_o\in \{0,1,...,d\}$ such that $\alpha_{i_0}$ is odd and the number of such $i_o$'s must also be odd. Hence, for a fixed $t$, the function $$Y\in B_{\frac{\epsilon}{\sqrt{t}}}\longmapsto\frac{1}{(2\pi)^{\frac{d}{2}}}\exp\left(-\frac{\vert\vert Y\vert\vert^2}{2}\right)c_{k,\alpha}(z)t^k(\sqrt{t}Y)^{\alpha}$$ is an odd function in $Y$ on the ball $B_\frac{\epsilon}{\sqrt{t}}$ and its integral on that ball is then $0$ by Lemma (\ref{simpleLemma}).

Now, let $G_i(t,\sqrt{t}Y)$ be the sum of all terms $T$ of the quantity $$\displaystyle \sum_{i=0}^N \frac{G^{(i)}_{(0,0)} \left[(t,\sqrt{t}Y),...,(t,\sqrt{t}Y)\right]}{i!}$$
that satisfy the following two conditions:
\begin{enumerate}[1.]
	\item $T$ is of degree $p\in \{0,1,2,...\}$ in $t$ and $T$ is of degree $q\in \{2k:\; k=0,1,2,...\}$ in the components of $\sqrt{t}Y = (\sqrt{t}Y_1,...,\sqrt{t}Y_d)$,
	\item $p$ and $q$ are such that $i=p+\frac{q}{2}$.
\end{enumerate}
Then there exists a polynomial $Q_i(Y)\in \mathbb{R}[Y]$ such that $$G_i(t,\sqrt{t}Y) = t^i\;Q_i(Y),$$ and we have

\begin{equation}\label{K_2_finalvalue}
	\begin{split}
K_2(t)&= \int_{\mathbb{R}^d}\frac{1}{(2\pi )^\frac{d}{2}}\exp\left(-\frac{\vert\vert Y\vert\vert^2}{2}\right)\sum_{i=0}^N \frac{G^{(i)}_{(0,0)} \left[(t,\sqrt{t}Y),...,(t,\sqrt{t}Y)\right]}{i!}\;dY + o(t^n)\\
&= \sum_{i=0}^nt^i\;\int_{\mathbb{R}^d}\frac{1}{(2\pi )^\frac{d}{2}}\exp\left(-\frac{\vert\vert Y\vert\vert^2}{2}\right) Q_i(Y)\;dY + o(t^n).
\end{split}
\end{equation} 
Furthermore,  the coefficients of $Q_i(Y)$ $(i=0,1,...,n)$ can be expressed as universal polynomials on the components of the  curvature tensor and its covariant derivatives at $z$.

\item \underline{\textbf{Computation of $K_3(t)$ (\ref{K_3})}}

Since $u_0(0,0)=1$, then $\sum_{i=0}^nu_i(0,y)t^i\geq \frac{1}{2}$ for small values of $t$. Hence $$\log \left[1+\frac{o(t^n)}{\sum_{i=0}^nu_it^i}\right]=\frac{o(t^n)}{\sum_{i=0}^nu_it^i}=o(t^n).$$  We then have
\begin{equation}\label{K_3_finalvalue}
K_3(t) = o(t^n)
\end{equation} because the function 
$$(t,y)\in \mathbb{R}\times U \longmapsto \left(\sum_{i=0}^nu_i(0,y)t^i \right)\sqrt{\mathrm{det}[g(y)]}\in \mathbb{R}$$
 that we integrate against the Gaussian distribution in (\ref{K_3}) is bounded for small values of $t$.
 
\end{enumerate}
By combining Equation (\ref{2ndEualityt}), Equation (\ref{K_1_finalvalue}), Equation (\ref{K_2_finalvalue}) and  Equation (\ref{K_3_finalvalue}), we have 

\begin{align*}
		\int_Uq_Z(t,z,w)\log q_Z(t,z,w)d\mathrm{Vol}(w)&=\sum_{i=0}^nt^i\; \int_{\mathbb{R}^d}\frac{1}{(2\pi )^\frac{d}{2}}\exp\left(-\frac{\vert\vert Y\vert\vert^2}{2}\right)\left(-P_i(Y)\;\frac{\vert\vert Y\vert\vert^2}{2}+Q_i(Y)\right)dY \\&\quad -\frac{d}{2}\log(2\pi t)+ o(t^n)
\end{align*}
The theorem is then proved by the formula in Equation (\ref{entropy_formula}).

\end{proof}

	\section{Computation of the coefficients $c_n(z)$}\label{lastsection}
In this section, we will describe how the coefficients $c_n(z)$ appearing in Theorem \ref{Higher-expansion} can actually be computed. For this, one needs the Taylor series at $0$ of $\sqrt{\mathrm{det}(g)(y)}$ and of the functions $u_i(0,y)$ ($i=0,1,...$) that appear in the parametrix expansion. The lowest order terms of these expansions have for instance been recorded by Sakai in \cite{sakai1971eigen}, and we will use these results. However, we do think it is insightful to see the general outline for computing these expansions, for which we present a general procedure in Subsection \ref{se:computation-general-procedure}. Afterwards, we compute the coefficient $c_n(z)$ of the asymptotic expansion in Theorem \ref{Higher-expansion} for $n = 0, 1, 2$.

\subsection{Computation of the Taylor series of relevant quantities}\label{se:computation-general-procedure}
In this subsection, we are going to describe how to compute the Taylor series at $y=0$ of $\sqrt{\mathrm{det}(g(y))}$ and of the functions $u_i(0,y)$ ($i=0,1,...$) appearing in the parametrix expansion. These are needed to compute the coefficients $c_i(z)$ $(i=0,1,2,...)$. 

\subsubsection{Taylor series of the metric coefficients $g_{ij}$}

We follow \cite{sakai1971eigen}. We let $y \in \mathbb{R}^d$, $\ell \in (0,\infty)$ and let $\gamma: \mathbb{R} \to M$ be a geodesic from $z$ such that for $k = 1, \dots, d$, its $k$th normal coordinate is
\[
\gamma^k(t) = \frac{y^k}{\ell} t.
\]
Then $\gamma^k(\ell) = y^k$. 
We define $J_i$ as the Jacobi field along the geodesic $\gamma$ with the initial conditions
\[
J_i(0) = 0
\]
and
\[
J_i'(0) = \frac{t}{\ell} \frac{\partial}{\partial y^i}(z).
\]
The Jacobi field is in fact given by
\[
J_i(t) = \frac{t}{\ell} \frac{\partial}{\partial y^i}(\gamma(t)).
\]
We set 
\[
f(t) := g( J_i(t), J_j(t) )
\]

Then $g_{ij}(y)$ equals $f(\ell)$. Therefore, to get the Taylor expansion for $g_{ij}$, it suffices to consider the Taylor expansion for $f$. To obtain this Taylor expansion, we use the Jacobi equation
\[
J_i''(t) = - R(\gamma'(t), J_i(t)) \gamma'(t)
\]
where the left-hand side should be interpreted as the second covariant derivative in the direction $\gamma'(t)$. By taking covariant derivatives of both sides, and (repeatedly) using that $\nabla_{\gamma'(t)} \gamma'(t) = 0$, we get that for all $k = 0, 1, \dots$, 
\[
J^{(k+2)}_i(t) = - \sum_{m=0}^k {k \choose m} R^{(m)}(\gamma'(t), J_i^{(k-m)}(t))\gamma'(t).
\]

\subsubsection{Taylor series of $g^{ij}$}

To get the Taylor series of $g^{ij}$ we use that for $g$ close enough to $I$,
\[
g^{-1} = I + (g - I) + (g - I)^2 + \cdots + (g - I)^n + .... 
\]

\subsubsection{Taylor series of $\det (g(y))$}

To obtain the Taylor series of $\det (g(y))$ we use the Jacobi formula
\[
\frac{d}{dt}\log\left( \det g(\gamma(t))\right) = \mathrm{tr}\left(g^{-1}(\gamma(t)) \frac{d}{dt} g(\gamma(t))\right)
\]
to first get a Taylor series for 
\[
\log (\det(g(\gamma(t))))
\]
after which we exponentiate this series.

\subsubsection{Taylor series of $u_0(0, y)$}

We know that

\[
u_0(0, y) = \left( \det(g(y)) \right)^{-1/4}
\]
and therefore to get the Taylor expansion of $u_0$ we use use the Taylor expansion of the function
\[
x \mapsto (1 + x)^{-1/4}.
\]

\subsubsection{Taylor series of $u_i(0, y)$}

Finally, the recursion formula for $u_{i}(0,y)$ is given by
\[
u_{i}(0, y) = \frac{1}{2} u_0(0, y) \int_0^1 \frac{\Delta u_{i-1}(0, s y)}{u_{0}(0, sy)}s^{i-1}ds
\]

Moreover, because every $u_i$ is smooth, the Taylor expansion up to some order $n$ for $\Delta u_i$ corresponds to the Laplacian of the Taylor expansion for $u_i$ up to order $n+2$.

\begin{remark}\label{importantRemark}
	Recall that we use stochastic normalization in Equation (\ref{heatequation}) (i.e. we use the operator $\frac{1}{2}\Delta$ instead of $\Delta$). Hence, due to the recurrence formula in Equation (1.4) of \cite{sakai1971eigen}, our heat invariant $u_i(0,y)$ ($y\in U$) is $\frac{1}{2^i}$times the one computed by Sakai \cite{sakai1971eigen} for $i\geq 1$, but the quantities $\sqrt{\mathrm{det}(g(y))}$ and $u_0(0, y)$ ($y\in U$) stay the same. 
\end{remark}

\subsection{Example for  $n=0,1$ and $2$}
We are going to compute $c_0(z), c_1(z)$ and $c_2(z)$. Let us first denote the Taylor expansions of $u_0(0,y),u_1(0,y),u_2(0,y)$ and $\sqrt{\mathrm{det}[g(y)]}$ as follows
\begin{equation}\label{taylorparamanddet}
	\begin{split}
	u_0(0,y)&= 1+A_2+A_3+A_4 + O(\vert\vert y\vert\vert^5),\\
	u_1(0,y)&= B_0+B_1+B_2+ O(\vert\vert y\vert\vert^3),\\
	u_2(0,y)&= C_0 + O(\vert\vert y\vert\vert),\\
	\sqrt{\mathrm{det}[g(y)]} & = 1+E_1+E_2+E_3+E_4+ O(\vert\vert y\vert\vert^5),
	\end{split}
\end{equation}
where for all integer $i\geq 0$, the terms $A_i,B_i,C_i$ and $E_i$ are homogeneous polynomials of degree $i$ in the components of $y$, i.e
 \begin{equation}\label{termtaylorexpansionNotation}
 	\begin{split}
 		A_i &= (A_i)_{k_1k_2...k_i}y^{k_1}y^{k_2}...y^{k_i}\\
 		B_i &= (B_i)_{k_1k_2...k_i}y^{k_1}y^{k_2}...y^{k_i}\\
 		C_i &= (C_i)_{k_1k_2...k_i}y^{k_1}y^{k_2}...y^{k_i}\\
 		E_i &= (E_i)_{k_1k_2...k_i}y^{k_1}y^{k_2}...y^{k_i}.\\
 	\end{split}
 \end{equation}
By \cite[Lemma 3.4 and Lemma 3.5]{sakai1971eigen}, we indeed have $A_0=E_0=1$ and $A_1=0.$ Furthermore, we have the following lemma.
\begin{lem}\label{sakaiLemma}
	With notation as in Equation (\ref{taylorparamanddet}) and Equation (\ref{termtaylorexpansionNotation}), we have
	\begin{enumerate}
		\renewcommand{\theenumi}{(\roman{enumi})}
		\item\label{it:comp-1} $\displaystyle E_1=0$
		\item\label{it:comp-2} $\displaystyle(E_2)_{ij}=-\frac{1}{6}\mathsf{Ric}_{ij}(z)$
		\item\label{it:comp-3} $\displaystyle(E_4)_{ijkl}=\frac{1}{144}\left(-\frac{18}{5}\mathsf{Ric}_{ij;kl}(z) + 2\;\mathsf{Ric}_{ij}(z)\mathsf{Ric}_{kl}(z) -\frac{4}{5}\mathsf{R}_{iujv}(z)\mathsf{R}_{kulv}(z) \right)$
		\item\label{it:comp-4} $\displaystyle(A_2)_{ij} = \frac{1}{12}\mathsf{Ric}_{ij}(z),$
		\item\label{it:comp-5} $\displaystyle(A_4)_{ijkl} = \frac{1}{24}\left(\frac{3}{10}\mathsf{Ric}_{ij;kl}(z)+ \frac{1}{12}\mathsf{Ric}_{ij}(z)\mathsf{Ric}_{kl}(z) + \frac{1}{15}\mathsf{R}_{iujv}(z)\mathsf{R}_{kulv}(z) \right),$
		\item\label{it:comp-6} $\displaystyle B_0 = \frac{1}{12}\mathsf{Sc}(z),$
		\item\label{it:comp-7} $\displaystyle (B_1)_i = \frac{1}{24}\mathsf{Sc}_{;i}(z),$
		\item\label{it:comp-8} $(B_2)_{ij}$ is given by
		\begin{align*}
			\displaystyle (B_2)_{ij} &= \frac{1}{720}\left(9\;\mathsf{Sc}_{;ij}(z)+ 3  \mathsf{Ric}_{ij;uu}(z) +5 \mathsf{Sc}(z)\mathsf{Ric}_{ij}(z)-4\mathsf{Ric}_{iu}(z)\mathsf{Ric}_{ju}(z) \right) \\
			&\quad + \frac{1}{720}\left[2\;\mathsf{Ric}_{uv}(z)\mathsf{R}_{iujv}(z)+2\;\mathsf{R}_{iuvw}(z)\mathsf{R}_{juvw}(z) \right].
		\end{align*}
		\item\label{it:comp-9} $C_0$ is given by
		\begin{align*}
			C_0 & = \frac{1}{1440}\left[9\;\mathsf{Sc}_{;ii}(z)+ 3\;  \mathsf{Ric}_{ii;uu}(z)+ 5\;\mathsf{Sc}^2(z) -4\;\mathsf{Ric}_{iu}(z)\mathsf{Ric}_{iu}(z)  \right]\\
			&\quad + \frac{1}{1440}\left[2\;\mathsf{Ric}_{uv}(z)\mathsf{R}_{iuiv}(z) +2\;\mathsf{R}_{iuvw}(z)\mathsf{R}_{iuvw}(z) \right].
		\end{align*}
	\end{enumerate}
\end{lem}
\begin{proof}
The values of $E_1,(E_2)_{ij}$ and $(E_4)_{ijkl}$ can be obtained by combining the Taylor series at $x=0$ of $(1+x)^{\frac{1}{2}}$ and the result of \cite[Lemma 3.4]{sakai1971eigen}.

The values of $(A_2)_{ij}$ and $(A_4)_{ijkl}$ can be obtained by combining \cite[Lemma 3.4]{sakai1971eigen} and Remark \ref{importantRemark}.

The values of $B_0, B_1$ and $B_2$ can be obtained by combining \cite[Equation (3.7) and Lemma 4.1]{sakai1971eigen} and Remark \ref{importantRemark}.

The value of $C_0$ can be obtained by combining \cite[Equation (3.8) and Lemma 4.1]{sakai1971eigen} and Remark \ref{importantRemark}.
\end{proof}

\subsubsection{Computation of $c_0(z)$ (case $n=0$)}
\begin{itemize}
\item \textbf{Step 1.}  By (\ref{thm:part2}) of Theorem \ref{Higher-expansion}, we need the functions
$$F(t,y)=\sqrt{\mathrm{det}[g(y)]}\;u_0(0,y)$$ and $$G(t,y) = F(t,y)\log[u_0(0,y)].$$ 
Here we only need the constant terms of the Taylor series of $F$ and $G$, i.e:
$$F_0(t,y) = F(0,0) = 1$$ and $$G_0(t,y) = G(0,0) = 0.$$

\item \textbf{Step 2.} Since $F_0(t,y)$ and $G_0(t,y)$ are constants, then there is no variable to replace as in the second step in part (\ref{thm:part2}) of Theorem \ref{Higher-expansion}.

\item \textbf{Step 3.} Since $P_0(Y) = 1$ and $Q_0(Y)=0,$ we have by (\ref{thm:part2}) of Theorem \ref{Higher-expansion}
\begin{align*}
c_0(z) &= \int_{\mathbb{R}^d}\frac{1}{(2\pi)^{\frac{d}{2}}}\exp\left(-\frac{\vert\vert Y\vert\vert^2}{2}\right)\left(-\frac{\vert\vert Y\vert\vert^2}{2} + 0\right)dY\\
&= -\frac{d}{2}.
\end{align*}
Therefore $$\boxed{c_0(z)=-\frac{d}{2}}.$$
\end{itemize}

\subsubsection{Computation of $c_1(z)$ (case $n=1$)}
\begin{enumerate}
\item \textbf{Step 1.} By (\ref{thm:part2}) of Theorem \ref{Higher-expansion}, we need the functions
$$F(t,y) = \sqrt{\mathrm{det}[g(y)]}[u_0(0,y)+ tu_1(0,y)]$$ and $$G(t,y) = F(t,y) \log[u_0(0,y)+ tu_1(0,y)].$$ We need the terms of the Taylor series at $(t=0,y=0)$ of $F$ and $G$ that are of degree zero in $t$ and of degree $2$ in $y$, or of degree one in $t$ and of degree $0$ in $y$. By using the Taylor series in (\ref{taylorparamanddet})5, we have:
\begin{align*}
F(t,y) &= \left(1+E_1+E_2 + O(\vert\vert y\vert\vert^3)\right)\left(1+A_2+ O(\vert\vert y\vert\vert^3) + tB_0 +tO(\vert\vert y\vert\vert)\right)\\
&= 1+A_2+tB_0+E_1+E_2+ ...
\end{align*}
where the remaining terms are either of degree strictly higher than two in $y$ or of degree more than one in both $t$ and $y$.
By using the Taylor expansion of the logarithm,
\begin{align*}
G(t,y) &= F(t,y) \log[u_0(0,y)+ tu_1(0,y)]\\
&= F(t,y)[tB_0+A_2+...]\\
&= -tB_0 -A_2 + ...
\end{align*}
where the remaining terms are either of degree strictly higher than two in $y$ or of degree more than one in both $t$ and $y$. Hence, we have $$F_1(t,y)=A_2+tB_0+E_2$$ and $$G_1(t,y)=-tB_0-A_2.$$
By Lemma 3.5 of \cite{sakai1971eigen}, we have $$\displaystyle A_2 = \frac{1}{12}\mathsf{Ric}_{kh}(z)y^ky^h.$$
By Lemma \ref{sakaiLemma} \ref{it:comp-6}, we have
$$\displaystyle B_0 = \frac{1}{12}\mathsf{Sc}(z).$$
and by Lemma \ref{sakaiLemma} \ref{it:comp-2} we have $$E_2=-\frac{1}{6}\mathsf{Ric}_{kh}(z)y^ky^h.$$

Therefore, we have:
\begin{align*}
F_1(t,y) &= \frac{1}{12}\mathsf{Ric}_{kh}(z)y^ky^h + \frac{t}{12}\mathsf{Sc}(z) -\frac{1}{6}\mathsf{Ric}_{kh}(z)y^ky^h\\
G_1(t,y)&= \frac{t}{12}\mathsf{Sc}(z) + \frac{1}{12}\mathsf{Ric}_{kh}(z)y^ky^h.
\end{align*}
\item \textbf{Step 2.} By doing the change of variables $y\longleftrightarrow \sqrt{t}Y$  [cf. (\ref{thm:part2}) of Theorem \ref{Higher-expansion}] we get
\begin{align*}
F_1(t,y) & = \frac{1}{12}\mathsf{Ric}_{kh}(z)\sqrt{t}Y^k\sqrt{t}Y^h + \frac{t}{12}\mathsf{Sc}(z) -\frac{1}{6}\mathsf{Ric}_{kh}(z)\sqrt{t}Y^k\sqrt{t}Y^h\\
&= \frac{t}{12}\mathsf{Ric}_{kh}(z)Y^kY^h + \frac{t}{12}\mathsf{Sc}(z) -\frac{t}{6}\mathsf{Ric}_{kh}(z)Y^kY^h,
\end{align*}
and 
\begin{align*}
G_1(t,y)&= \frac{t}{12}\mathsf{Sc}(z) + \frac{1}{12}\mathsf{Ric}_{kh}(z)\sqrt{t}Y^k\sqrt{t}Y^h\\
& = \frac{t}{12}\mathsf{Sc}(z) + \frac{t}{12}\mathsf{Ric}_{kh}(z)Y^kY^h.
\end{align*}
Therefore

$$P_1(Y)=\frac{1}{12}\mathsf{Sc}(z) -\frac{1}{12}\mathsf{Ric}_{kh}(z)Y^kY^h $$
and $$Q_1(Y)=\frac{1}{12}\mathsf{Sc}(z) + \frac{1}{12}\mathsf{Ric}_{kh}(z)Y^kY^h.$$

\item \textbf{Step 3.} We are now ready to compute a Gaussian integral to get $c_1(z)$. By using Lemma \ref{lemma1} (\ref{lemma1:ii}), we have
\begin{align*}
\int_{\mathbb{R}^d}\frac{1}{(2\pi)^{\frac{d}{2}}}\exp\left(-\frac{\vert\vert Y\vert\vert^2}{2}\right)P_1(Y)\frac{\vert\vert Y\vert\vert^2}{2} dY &= \frac{d}{24}\mathsf{Sc}(z)- \frac{d+2}{24}\mathsf{Sc}(z)\\
&= -\frac{1}{12}\mathsf{Sc}(z)
\end{align*}
and 
\begin{align*}
	\int_{\mathbb{R}^d}\frac{1}{(2\pi)^{\frac{d}{2}}}\exp\left(-\frac{\vert\vert Y\vert\vert^2}{2}\right)Q_1(Y)dY &= \frac{1}{12}\mathsf{Sc}(z) + \frac{1}{12}\mathsf{Sc}(z)\\
	& = \frac{1}{6} \mathsf{Sc}(z).
\end{align*}
Therefore
$$\boxed{\displaystyle c_1(z) = -\int_{\mathbb{R}^d}\frac{1}{(2\pi)^{\frac{d}{2}}}\exp\left(-\frac{\vert\vert Y\vert\vert^2}{2}\right)P_1(Y)\frac{\vert\vert Y\vert\vert^2}{2} dY+\int_{\mathbb{R}^d}\frac{1}{(2\pi)^{\frac{d}{2}}}\exp\left(-\frac{\vert\vert Y\vert\vert^2}{2}\right)Q_1(Y)dY = \frac{1}{4}\mathsf{Sc}(z)}$$
\end{enumerate}

We have therefore proven rigorously the following theorem, which is Proposition $1$ of \cite{ijcai2020p375}.

\begin{theorem}\label{entropy}
	Let $z\in Z$ and let $\mathsf{Sc}(z)$ denotes the scalar curvature at $z$. Then:
	$$D_{KL}\left(q_Z(t,z,.)\vert\vert \vartheta\right) = -\frac{d}{2}\log(2\pi t) -\frac{d}{2}+\log[\mathrm{Vol}(Z)]+\frac{1}{4}\mathsf{Sc}(z)\;t +o(t),$$
	where $\vartheta$ is the normalized Riemannian volume on $Z$.
\end{theorem}

\subsubsection{Computation of $c_2(z)$ (case $n=2$)}
In this last subsection, we are going to prove the following theorem.
\begin{theorem}\label{theorem_c_2}
	With the notations of Theorem \ref{Higher-expansion}, we have:
	\begin{equation}
		\begin{split}
			c_2(z)&= -\frac{3d-22}{480}\;\mathsf{Sc}_{;ii}(z)-\frac{1}{48}\;\mathsf{Ric}_{ij}(z)\mathsf{Ric}_{ij}(z)+\frac{d+6}{80}\;\mathsf{Ric}_{ij;ij}(z)\\
			&\quad -\frac{d-14}{1440}\;\mathsf{R}_{iujv}(z)\mathsf{R}_{iujv}(z)+\frac{d+6}{720}\;\mathsf{R}_{iujv}(z)\mathsf{R}_{juiv}(z).
		\end{split}
	\end{equation}
\end{theorem}
\begin{proof}
We are going to use the method outlined in Theorem \ref{Higher-expansion}.
\begin{enumerate}
\item \textbf{Step 1.} We have the two functions
$$F(t,y) = \sqrt{\mathrm{det}[g(y)]}[u_0(0,y)+ tu_1(0,y)+ t^2u_2(0,y)]$$ and $$G(t,y) = F(t,y) \log[u_0(0,y)+ tu_1(0,y)+t^2u_2(0,y)].$$ We can get terms in the Taylor series of $F$ and $G$ by using the Taylor expansion in Equations  \ref{taylorparamanddet}: terms without $t$ should be of degree at most $4$ in the components of $y$, terms with $t$ should be of degree at most $2$ in the components of $y$ and terms with $t^2$ should be of degree $0$ in the components of $y$. We have (using that $E_1 = 0$ as stated in Lemma \ref{sakaiLemma} \ref{it:comp-1})
\begin{align*}
F(t,y) &= \left(1+E_2+E_3+E_4+\dots\right)\;\left(1+A_2+A_3+A_4+tB_0+tB_1+tB_2+t^2C_0+\dots \right)\\
&= 1+A_2+A_3+A_4+tB_0+tB_1+tB_2+t^2C_0+ \\
&\quad +E_2+ E_2A_2+tE_2B_0+E_3+E_4+...\\
&= 1+ A_2 + t B_0 + E_2+ A_3 + tB_1+E_3\\&\quad+ A_4 + t B_2+t^2C_0+E_2A_2+  tE_2B_0 +E_4 + \dots
\end{align*}
By using the Taylor series at $0$ of $\log(1+x)$, we get the expansion
\begin{align*}
\log&\left(u_0(0,y)+tu_1(0,y)+t^2u_2(0,y) \right)
\\&= \log\left(1+A_2+A_3+A_4+tB_0+tB_1+tB_2+t^2C_0+\dots \right)\\
&= A_2+A_3+A_4+tB_0+tB_1+tB_2+t^2C_0-\frac{1}{2}(A_2+tB_0+ \dots)^2+\dots\\
&= A_2+A_3+A_4+tB_0+tB_1+tB_2+t^2C_0-\frac{1}{2}A_2^2 -tB_0A_2 -\frac{1}{2}t^2B_0^2+\dots\\
&=A_2 + tB_0+A_3 + tB_1+A_4+tB_2+t^2C_0 -\frac{1}{2}A_2^2 -tB_0A_2 -\frac{1}{2}t^2B_0^2 +\dots
\end{align*}

By combining  this with the expression of $G$, we have
\begin{align*}
G(t,y)&= F(t,y)\log[u_0+tu_1+t^2u_2+ \dots]\\
&= F(t,y) [A_2 + tB_0+A_3 + tB_1+A_4+tB_2+t^2C_0 -\frac{1}{2}A_2^2 -tB_0A_2 -\frac{1}{2}t^2B_0^2+\dots]\\
&= A_2 + tB_0+A_3 + tB_1+A_4+tB_2+t^2C_0 -\frac{1}{2}A_2^2 -tB_0A_2 -\frac{1}{2}t^2B_0^2\\
&\quad + A_2^2 + t A_2 B_0 + t B_0 A_2 + t^2 B_0^2 + E_2 A_2 + t E_2 B_0+ \dots\\
&= A_2 + tB_0+A_3 + tB_1+A_4+tB_2+t^2C_0 + \frac{1}{2}A_2^2 + tB_0A_2 +\frac{1}{2}t^2B_0^2\\
&\quad + E_2 A_2 + t E_2 B_0+ \dots
\end{align*}

Hence, we find
\begin{align*}
	F_2(t,y)&= A_4+tB_2+t^2C_0+E_2A_2+tE_2B_0+ E_4
\end{align*}
and 
\begin{align*}
	G_2(t,y)&=A_4+tB_2+t^2C_0 + \frac{1}{2}A_2^2 + tB_0A_2 +\frac{1}{2}t^2B_0^2 + E_2 A_2 + t E_2 B_0.
\end{align*}
\item \textbf{Step 2.} By using the change of variables $y\longleftrightarrow \sqrt{t}Y$, by using the notations in Equation (\ref{termtaylorexpansionNotation}) we have the two polynomials
\begin{align*}
P_2(Y)&=C_0+(B_2)_{ij}Y^iY^j+B_0(E_2)_{ij}Y^iY^j+ (A_4)_{ijkl}Y^iY^jY^kY^l+(E_2)_{ij}(A_2)_{kl}Y^iY^jY^kY^l\\
&\quad+ (E_4)_{ijkl}Y^iY^jY^kY^l,
\end{align*}
and

\begin{align*}
Q_2(Y)&=C_0 + \frac{1}{2}B_0^2+ (B_2)_{ij}Y^iY^j + B_0(A_2)_{ij}Y^iY^j + B_0(E_2)_{ij}Y^iY^j + (A_4)_{ijkl}Y^iY^jY^kY^l\\
&\quad  + \frac{1}{2}(A_2)_{ij} (A_2)_{kl}Y^iY^jY^kY^l+ (E_2)_{ij}(A_2)_{jk}Y^iY^jY^kY^l.
\end{align*}
\item \textbf{Step 3.} All we need to do now is to compute Gaussian integrals. By using Lemma \ref{contractionlemma}, we find
\begin{align*}
	c_2(z) &= \int_{\mathbb{R}^d}\frac{1}{(2\pi)^{\frac{d}{2}}}\exp\left(-\frac{\vert\vert Y\vert\vert^2}{2}\right)\left(-\frac{\vert\vert Y\vert\vert^2}{2}P_2(Y)+Q_2(Y) \right)dY\\
	&=- \frac{d}{2} C_0 - \frac{d+2}{2} (B_2)_{ii} - \frac{d+2}{2} B_0 (E_2)_{ii}
	  - \frac{d+4}{2} \left[(A_4)_{iijj} + (A_4)_{ijij} + (A_4)_{ijji}\right]\\
	& \quad - \frac{d+4}{2} \left[(E_2)_{ii}(A_2)_{jj} + (E_2)_{ij}(A_2)_{ij}
	  						+(E_2)_{ij}(A_2)_{ji}\right]\\
	& \quad - \frac{d+4}{2} \left[ (E_4)_{iijj} + (E_4)_{ijij} + (E_4)_{ijji} \right]\\
	& \quad+ C_0 + \frac{1}{2}(B_0)^2 + (B_2)_{ii} + B_0(A_2)_{ii} + B_0(E_2)_{ii} +
	  (A_4)_{iijj} + (A_4)_{ijij} + (A_4)_{ijji}\\
	  &\quad + \frac{1}{2} \left[ (A_2)_{ii}(A_2)_{jj} + 
	                              (A_2)_{ij}(A_2)_{ij} +
	                              (A_2)_{ij}(A_2)_{ji} \right]\\
	&\quad  + (E_2)_{ii}(A_2)_{jj} + (E_2)_{ij} (A_2)_{ij} + (E_2)_{ij}(A_2)_{ji}
\end{align*}

which simplifies to
\begin{equation}
\begin{split}
	\label{c_2all}
	c_2(z) &= - \frac{d-2}{2} C_0 - \frac{d}{2} (B_2)_{ii} - \frac{d}{2}B_0(E_2)_{ii} 
			  - \frac{d+2}{2}\left[ (A_4)_{iijj} + (A_4)_{ijij} + (A_4)_{ijji} \right]	\\
	&\quad - \frac{d+2}{2} \left[(E_2)_{ii}(A_2)_{jj} + (E_2)_{ij}(A_2)_{ij}
	+(E_2)_{ij}(A_2)_{ji} \right]\\
	&\quad - \frac{d+4}{2} \left[ (E_4)_{iijj} + (E_4)_{ijij} + (E_4)_{ijji} \right]\\
	&\quad + \frac{1}{2} B_0^2 + B_0(A_2)_{ii} + \frac{1}{2} \left[ (A_2)_{ii}(A_2)_{jj} + 
	(A_2)_{ij}(A_2)_{ij} +
	(A_2)_{ij}(A_2)_{ji} \right].
\end{split}
\end{equation}

\end{enumerate}
Let us compute each term of Equation (\ref{c_2all}) by using the results of Lemma \ref{sakaiLemma}.
\begin{itemize}
\item By \ref{it:comp-6} and \ref{it:comp-9} of Lemma \ref{sakaiLemma} and by using the fact that $\mathsf{R}_{iuiv}= \mathsf{R}_{uivi}= \mathsf{Ric}_{uv}$  \cite[Equation (D.I.4)]{berger1971spectre}, we have
\begin{equation}\label{termB_0_C_0_squarre}
\begin{split}
-\frac{d-2}{2}C_0 + \frac{1}{2}B_0^2&= -\frac{d-2}{2880}\left[9\;\mathsf{Sc}_{;ii}(z)+ 3\;  \mathsf{Ric}_{ii;uu}(z) -4\;\mathsf{Ric}_{iu}(z)\mathsf{Ric}_{iu}(z) +2\;\mathsf{Ric}_{uv}(z)\mathsf{R}_{iuiv}(z) \right]\\
&\quad - \frac{d-2}{2880}\left[ 2\;\mathsf{R}_{iuvw}(z)\mathsf{R}_{iuvw}(z) + 5\;\mathsf{Sc}^2(z)\right] + \frac{1}{288}\mathsf{Sc}^2(z)\\
&= -\frac{d-2}{2880}\left[12\;\mathsf{Sc}_{;ii}(z) -2\;\mathsf{Ric}_{iu}(z)\mathsf{Ric}_{iu}(z)+ 2\;\mathsf{R}_{iuvw}(z)\mathsf{R}_{iuvw}(z) \right]\\
&\quad -\frac{d-4}{576}\;\mathsf{Sc}^2(z)
\end{split}
\end{equation}
\item By using \ref{it:comp-8} of Lemma \ref{sakaiLemma} and by using the fact that $\mathsf{R}_{iuiv}= \mathsf{R}_{uivi}= \mathsf{Ric}_{uv}$  (cf. \cite[Equation (D.I.4)]{berger1971spectre}), we have
\begin{equation}\label{termB_2}
\begin{split}
-\frac{d}{2}(B_2)_{ii}&= -\frac{d}{1440}\left(9\;\mathsf{Sc}_{;ii}(z)+ 3  \mathsf{Ric}_{ii;uu}(z) + 5 \mathsf{Sc}(z)\mathsf{Ric}_{ii}(z)-4\mathsf{Ric}_{iu}(z)\mathsf{Ric}_{iu}(z) \right) \\
&\quad - \frac{d}{1440}\left(2\mathsf{Ric}_{uv}(z)\mathsf{R}_{iuiv}(z)+2\;\mathsf{R}_{iuvw}(z)\mathsf{R}_{iuvw}(z) \right)\\
&= -\frac{d}{1440}\left[12\;\mathsf{Sc}_{;ii}(z)+5\; \mathsf{Sc}^2(z)-2\mathsf{Ric}_{iu}(z)\mathsf{Ric}_{iu}(z) \right]\\
&\quad -\frac{2d}{1440}\;\mathsf{R}_{iuvw}(z)\mathsf{R}_{iuvw}(z)
\end{split}
\end{equation}
\item By using \ref{it:comp-3} and \ref{it:comp-5} of Lemma \ref{sakaiLemma}, we have the term of Equation (\ref{c_2all}) with $A_4$ and $E_4$
\begin{equation}\label{termswithA_4E_4}
\begin{split}
&-\frac{d+2}{2}[(A_4)_{ijij}+(A_4)_{iijj}+(A_4)_{ijji}]-\frac{d+4}{2}[(E_4)_{ijij}+(E_4)_{iijj}+(E_4)_{ijji}]\\
&=-\frac{d+2}{48}\left(\frac{3}{10}\mathsf{Ric}_{ij;ij}(z)+ \frac{1}{12}\mathsf{Ric}_{ij}(z)\mathsf{Ric}_{ij}(z) + \frac{1}{15}\mathsf{R}_{iujv}(z)\mathsf{R}_{iujv}(z) \right)\\
&\quad -\frac{d+2}{48}\left(\frac{3}{10}\mathsf{Ric}_{ii;jj}(z)+ \frac{1}{12}\mathsf{Ric}_{ii}(z)\mathsf{Ric}_{jj}(z) + \frac{1}{15}\mathsf{R}_{iuiv}(z)\mathsf{R}_{jujv}(z) \right)\\
&\quad-\frac{d+2}{48}\left(\frac{3}{10}\mathsf{Ric}_{ij;ji}(z)+ \frac{1}{12}\mathsf{Ric}_{ij}(z)\mathsf{Ric}_{ji}(z) + \frac{1}{15}\mathsf{R}_{iujv}(z)\mathsf{R}_{juiv}(z) \right)\\
&\quad -\frac{d+4}{288}\left(-\frac{18}{5}\mathsf{Ric}_{ij;ij}(z) + 2\;\mathsf{Ric}_{ij}(z)\mathsf{Ric}_{ij}(z) -\frac{4}{5}\mathsf{R}_{iujv}(z)\mathsf{R}_{iujv}(z) \right)\\
&\quad -\frac{d+4}{288}\left(-\frac{18}{5}\mathsf{Ric}_{ii;jj}(z) + 2\;\mathsf{Ric}_{ii}(z)\mathsf{Ric}_{jj}(z) -\frac{4}{5}\mathsf{R}_{iuiv}(z)\mathsf{R}_{jujv}(z) \right)\\
&\quad -\frac{d+4}{288}\left(-\frac{18}{5}\mathsf{Ric}_{ij;ji}(z) + 2\;\mathsf{Ric}_{ij}(z)\mathsf{Ric}_{ji}(z) -\frac{4}{5}\mathsf{R}_{iujv}(z)\mathsf{R}_{juiv}(z) \right)\\
&\overset{\textbf{(*)}}{=} -\frac{d+2}{48}\left(\frac{3}{10}\mathsf{Ric}_{ij;ij}(z)+ \frac{1}{12}\mathsf{Ric}_{ij}(z)\mathsf{Ric}_{ij}(z) + \frac{1}{15}\mathsf{R}_{iujv}(z)\mathsf{R}_{iujv}(z) \right)\\
&\quad -\frac{d+2}{48}\left(\frac{3}{10}\mathsf{Sc}_{;jj}(z)+ \frac{1}{12}\mathsf{Sc}^2(z) + \frac{1}{15}\mathsf{Ric}_{uv}(z)\mathsf{Ric}_{uv}(z) \right)\\
&\quad-\frac{d+2}{48}\left(\frac{3}{10}\mathsf{Ric}_{ij;ij}(z)+ \frac{1}{12}\mathsf{Ric}_{ij}(z)\mathsf{Ric}_{ij}(z) + \frac{1}{15}\mathsf{R}_{iujv}(z)\mathsf{R}_{juiv}(z) \right)\\
&\quad -\frac{d+4}{288}\left(-\frac{18}{5}\mathsf{Ric}_{ij;ij}(z) + 2\;\mathsf{Ric}_{ij}(z)\mathsf{Ric}_{ij}(z) -\frac{4}{5}\mathsf{R}_{iujv}(z)\mathsf{R}_{iujv}(z) \right)\\
&\quad -\frac{d+4}{288}\left(-\frac{18}{5}\mathsf{Sc}_{;jj}(z) + 2\;\mathsf{Sc}^2(z) -\frac{4}{5}\mathsf{Ric}_{uv}(z)\texttt{R}_{uv}(z) \right)\\
&\quad -\frac{d+4}{288}\left(-\frac{18}{5}\mathsf{Ric}_{ij;ij}(z) + 2\;\mathsf{Ric}_{ij}(z)\mathsf{Ric}_{ij}(z) -\frac{4}{5}\mathsf{R}_{iujv}(z)\mathsf{R}_{juiv}(z) \right)\\
&= \frac{d+6}{160}\left(2 \; \mathsf{Ric}_{ij;ij}(z)+\mathsf{Sc}_{;jj}(z) \right)-\frac{23d+78}{1440}\mathsf{Ric}_{ij}(z)\mathsf{Ric}_{ij}(z)\\
&\quad -\frac{5d+18}{576}\mathsf{Sc}^2(z)+\frac{d+6}{720}\left(\mathsf{R}_{iujv}(z)\mathsf{R}_{iujv}(z)+\mathsf{R}_{iujv}(z)\mathsf{R}_{juiv}(z) \right),
\end{split}
\end{equation}
where \textbf{(*)} follows since $\mathsf{Ric}_{ij} = \mathsf{Ric}_{ji}$.
\item By using \ref{it:comp-2}, \ref{it:comp-4} and \ref{it:comp-6} of Lemma \ref{sakaiLemma}, we get for the remaining terms
\begin{equation}\label{remainingtermofc_2all}
\begin{split}
&-\frac{d}{2}B_0(E_2)_{ii} + B_0(A_2)_{ii}-\frac{d+2}{2}\left((E_2)_{ij}(A_2)_{ij}+(E_2)_{ii}(A_2)_{jj}+(E_2)_{ij}(A_2)_{ji} \right)\\
&\quad+\frac{1}{2}\left[ (A_2)_{ii}(A_2)_{jj} + (A_2)_{ij}(A_2)_{ij} + (A_2)_{ij}(A_2)_{ji} \right]\\
&= \frac{d}{144}\mathsf{Sc}^2(z)+ \frac{1}{144}\mathsf{Sc}^2(z)+\frac{d+2}{2}\left(\frac{1}{36}\;\mathsf{Ric}_{ij}(z)\mathsf{Ric}_{ij}(z)+\frac{1}{72}\mathsf{Ric}_{ii}(z)\mathsf{Ric}_{ii}(z) \right)\\
&\quad+\frac{1}{288}\left[ \mathsf{Sc}^2(z) + 2 \; \mathsf{Ric}_{ij}\mathsf{Ric}_{ij} \right]\\
&= \frac{4d+7}{288}\;\mathsf{Sc}^2(z) +\frac{2d+5}{144}\;\mathsf{Ric}_{ij}(z)\;\mathsf{Ric}_{ij}(z).
\end{split}
\end{equation}

\end{itemize}
Finally, by combining Equation (\ref{c_2all}), Equation (\ref{termB_0_C_0_squarre}), Equation (\ref{termB_2}), Equation (\ref{termswithA_4E_4}) and Equation (\ref{remainingtermofc_2all}), we have
\begin{align*}
c_2(z)&= -\frac{d-2}{2880}\left[12\;\mathsf{Sc}_{;ii}(z) -2\;\mathsf{Ric}_{iu}(z)\mathsf{Ric}_{iu}(z)+ 2\;\mathsf{R}_{iuvw}(z)\mathsf{R}_{iuvw}(z) \right]\\
&\quad - \frac{5(d-4)}{2880}\;\mathsf{Sc}^2(z)\\
&\quad -\frac{2 d}{2\cdot1440}\left[12\;\mathsf{Sc}_{;ii}(z)+5\; \mathsf{Sc}^2(z)-2\;\mathsf{Ric}_{iu}(z)\mathsf{Ric}_{iu}(z)+2\;\mathsf{R}_{iuvw}(z)\mathsf{R}_{iuvw}(z) \right]  \\
&\quad+ \frac{18(d+6)}{18 \cdot 160}\left[2\;\mathsf{Ric}_{ij;ij}(z)+\mathsf{Sc}_{;jj}(z) \right]-\frac{2(23d+78)}{2\cdot 1440}\mathsf{Ric}_{ij}(z)\mathsf{Ric}_{ij}(z)\\
&\quad -\frac{5(5d+18)}{5 \cdot 576}\mathsf{Sc}^2(z)+\frac{4(d+6)}{4 \cdot 720}\left(\mathsf{R}_{iujv}(z)\mathsf{R}_{iujv}(z)+\mathsf{R}_{iujv}(z)\mathsf{R}_{juiv}(z) \right)\\
&\quad +\frac{10(4d+7)}{10 \cdot 288}\;\mathsf{Sc}^2(z) +\frac{20(2d+5)}{20 \cdot 144}\;\mathsf{Ric}_{ij}(z)\;\mathsf{Ric}_{ij}(z)\\
&= -\frac{3d-22}{480}\;\mathsf{Sc}_{;ii}(z)-\frac{1}{48}\;\mathsf{Ric}_{ij}(z)\mathsf{Ric}_{ij}(z)+\frac{d+6}{80}\;\mathsf{Ric}_{ij;ij}\\
&\quad -\frac{d-14}{1440}\;\mathsf{R}_{iujv}(z)\mathsf{R}_{iujv}(z)+\frac{d+6}{720}\;\mathsf{R}_{iujv}(z)\mathsf{R}_{juiv}(z),
\end{align*}
and this ends the proof of Theorem \ref{theorem_c_2}.
\end{proof}

	\newpage	
	\renewcommand{\bibname}{References}
	\nocite{*}
	\bibliographystyle{plainnat}
	\bibliography{references}
	\addcontentsline{toc}{section}{References}
	\setcitestyle{square}
	
\end{document}